\documentclass[a4paper,11pt,twoside]{amsart}
\usepackage[english]{babel}
\usepackage[utf8]{inputenc}

\usepackage[a4paper,inner=3cm,outer=3cm,top=4cm,bottom=4cm,pdftex]{geometry}
\usepackage{fancyhdr}
\usepackage{titlesec}
\usepackage{color}
\usepackage{bold-extra}
\usepackage{ mathrsfs }
\titleformat{\section}{\normalfont\scshape\centering}{\thesection}{1em}{}
  \titleformat{\subsection}{\bfseries}{\thesubsection}{1em}{}

\usepackage{comment}
\usepackage{graphics}
\usepackage{aliascnt}
\usepackage[pdftex,citecolor=green,linkcolor=red]{hyperref}

\usepackage{enumerate}
\usepackage{amsmath}
\usepackage{amsfonts}
\usepackage{amssymb}
\usepackage{amsthm}
\usepackage{comment}
\usepackage{mathtools}

\newtheorem{theorem}{Theorem}[section]
\newtheorem{corollary}[theorem]{Corollary}

\newtheorem{lemma}[theorem]{Lemma}

\theoremstyle{definition}
\newtheorem{definition}[theorem]{Definition}
\newtheorem{remark}[theorem]{Remark}
\newtheorem{conjecture}[theorem]{Conjecture}

\numberwithin{equation}{section}

\renewcommand\leq{\leqslant}
\renewcommand\geq{\geqslant}
\newcommand{\sgn}{\text{sgn}}

\renewcommand{\Re}{\textnormal{Re}}

\newcommand\E{\mathbb{E}}

\newcommand\vol{\textnormal{vol}}
\newcommand{\m}{\mathrm{mod}\ }
\newcommand{\Mod}[1]{\ (\mathrm{mod}\ #1)}

\renewcommand\d{\,\mathrm{d}}
\begin{document}
\title{On a Bohr set analogue of Chowla's conjecture}

\author{Joni Ter\"{a}v\"{a}inen}
\address{Department of Mathematics and Statistics, University of Turku, 20014 Turku, Finland}
\email{joni.p.teravainen@gmail.com}

\author[Aled Walker]{Aled Walker}
\address{Department of Mathematics, King's College London, London WC2R 2LS, United Kingdom}
\email{aled.walker@kcl.ac.uk}

\begin{abstract} 
Let $\lambda$ denote the Liouville function. We show that the logarithmic mean of $\lambda(\lfloor \alpha_1n\rfloor)\lambda(\lfloor \alpha_2n\rfloor)$ is $0$ whenever $\alpha_1,\alpha_2$ are positive reals with $\alpha_1/\alpha_2$ irrational.  We also show that for $k\geq 3$ the logarithmic mean of $\lambda(\lfloor \alpha_1n\rfloor)\cdots \lambda(\lfloor \alpha_kn\rfloor)$ has some nontrivial amount of cancellation, under certain rational independence assumptions on the real numbers $\alpha_i$. Our results for the Liouville function generalise to produce independence statements for general bounded real-valued multiplicative functions evaluated at Beatty sequences. These results answer the two-point case of a conjecture of Frantzikinakis (and provide some progress on the higher order cases), generalising a recent result of  Crn\v{c}evi\'c--Hern\'andez--Rizk--Sereesuchart--Tao. 

As an ingredient in our proofs, we establish bounds for the logarithmic correlations of the Liouville function along Bohr sets.
\end{abstract}

\maketitle
\section{Introduction}

Let $\lambda:\mathbb{N}\to \{-1,+1\}$ denote the Liouville function: that is, the completely multiplicative function with $\lambda(p)=-1$ for all primes $p$.  In this note, we consider correlations of the Liouville function (as well as arbitrary multiplicative functions) along Beatty sequences $\lfloor \alpha n\rfloor$. 

For correlations of `length 1' (i.e. single averages of $\lambda$ over Beatty sequences), it  follows from a classical exponential sum estimate of Davenport\footnote{Indeed, by Davenport's result, $\sum_{n\leq X}\lambda(n)e(\beta n)=o(X)$ for all $\beta$. If $\alpha$ is rational, the claim follows easily from this. If $\alpha$ is irrational, by considering the sums $\sum_{n \leqslant X} (1 \pm \lambda(n)) e(k\alpha n)$ and applying Weyl's criterion, the sequence $\{\alpha n: \lambda(n)=v\}$ is uniformly distributed modulo $1$ for $v\in \{-1+1\}$. But now if $\alpha>1$ then $\sum_{n\leq X}\lambda(\lfloor \alpha n\rfloor)=\sum_{m\leq \alpha X, m/\alpha \in [1-1/\alpha,1)\pmod 1}\lambda(m)$, and by the uniform distribution property mentioned above this is $o(X)$. The case $\alpha \in (0,1)$ follows along similar lines.}~\cite{davenport} that for all $\alpha >0$ \[\lim_{X \rightarrow \infty} \frac{1}{X}\sum_{n \leqslant X} \lambda(\lfloor \alpha n \rfloor) = 0.\] The following far-reaching extension was posed as an open problem by Frantzikinakis\footnote{Special case of \cite[Problem 2]{Fr20}, see remark following this problem. Also stated by Frantzikinakis in a talk at Additive Combinatorics Webinar, July 2020.}.

\begin{conjecture} 
\label{conj_frantz}
Let $k\geq 1$ be an integer, and let $\alpha_1,\ldots, \alpha_k>0$ be such that $1,\alpha_1,\ldots, \alpha_k$ are linearly independent over $\mathbb{Q}$. Then, for any multiplicative functions $f_1,\ldots, f_k:\mathbb{N}\to [-1,1]$, we have
\begin{align}\label{eq_BohrElliott}
\lim_{X\to \infty}\mathbb{E}_{n\leq X}^{\log}\prod_{i=1}^k f_i(\lfloor \alpha_i n\rfloor)=\prod_{i=1}^k \lim_{X\to \infty}\mathbb{E}_{n\leq X}^{\log}f_i(n).  
\end{align}
In particular, we have
\begin{align}\label{eq_BohrChowla}
\lim_{X\to \infty}\mathbb{E}_{n\leq X}^{\log}\lambda(\lfloor \alpha_1 n\rfloor)\cdots \lambda(\lfloor \alpha_k n\rfloor)=0.    
\end{align}
\end{conjecture}
\noindent Here and throughout, $\E_{n \leqslant X}^{\log} f(n)$ denotes the logarithmic average $\frac{1}{\log X} \sum_{n \leqslant X} \frac{f(n)}{n}$. We use $\E_{n \leqslant X} f(n)$ to denote the natural average $\frac{1}{X} \sum_{n \leqslant X} f(n)$. 

\textbf{Remarks.}
\begin{itemize}
    \item The limits on the right-hand side of \eqref{eq_BohrElliott} always exist, since by Wirsing's theorem~\cite[Theorem 4.6 in Section III.4]{tenenbaum} any bounded, real-valued multiplicative function has a mean value. 

    \item The claim~\eqref{eq_BohrChowla} should hold more generally when $\alpha_i/\alpha_j$ is irrational for all $i\neq j$, but~\eqref{eq_BohrElliott} does not hold under this weaker assumption (for a counterexample, take $k=2$, $f_1(n)=f_2(n)=1_{(n,2)=1}$ and $\alpha_1=\sqrt{2}, \alpha_2=\sqrt{2}+2$).
\end{itemize}

For $k=2$, Conjecture~\ref{conj_frantz} was recently proved in~\cite[Theorem B]{CHRST} by Crn\v{c}evi\'c--Hern\'andez--Rizk--Sereesuchart--Tao, under the additional assumption that $\alpha_1 = 1$. Conjecture~\ref{conj_frantz} for $k=2$ was also posed in a more general setting of ``bounded multiplicative approximately invariant sequences'' as~\cite[Conjecture 5.1]{CHRST}, but we will only consider multiplicative functions in this note. One may also consult~\cite[Conjecture 5.2]{CHRST} to see the Liouville case of Conjecture~\ref{conj_frantz} in print when $\alpha_1 = 1$. 

Our first main theorem settles Conjecture~\ref{conj_frantz} when $k=2$, for arbitrary $\alpha_1,\alpha_2$. More generally, the following result applies to two-point correlations of bounded multiplicative functions along inhomogeneous Beatty sequences $\lfloor \alpha n+\beta\rfloor$. In the case of the Liouville function, it gives a complete characterisation of when such correlations converge to $0$.

\begin{theorem}[Two-point correlations along Beatty sequences]\label{thm_twopoint}
Let $\alpha_1,\alpha_2>0$ and $\beta_1,\beta_2\in \mathbb{R}$. Let $f_1,f_2:\mathbb{N}\to [-1,1]$ be multiplicative functions.
\begin{enumerate}
    \item Suppose that $1, \alpha_1,\alpha_2$ are linearly independent over $\mathbb{Q}$. Then\footnote{Here and in what follows, we extend multiplicative functions defined on $\mathbb{N}$ arbitrarily to $\mathbb{Z}$.}
     \begin{align*}
\lim_{X\to \infty}\mathbb{E}^{\log}_{n\leq X}  f_1(\lfloor \alpha_1 n+\beta_1\rfloor)f_2(\lfloor \alpha_2 n+\beta_2\rfloor) =\lim_{X\to \infty}\Big(\mathbb{E}^{\log}_{n\leq X}  f_1(n)\Big)\cdot \lim_{X\to \infty}\Big(\mathbb{E}^{\log}_{n\leq X}  f_2(n)\Big). 
\end{align*}

\item Suppose that $\alpha_1/\alpha_2$ is irrational. Then we have 
    \begin{align*}
\lim_{X\to \infty}\mathbb{E}^{\log}_{n\leq X}  \lambda(\lfloor \alpha_1 n+\beta_1\rfloor)\lambda(\lfloor \alpha_2 n+\beta_2\rfloor) =0. 
\end{align*}

    \item Suppose that $r:=\alpha_1/\alpha_2$ is rational. Then 
\begin{align*}
\lim_{X\to \infty}\mathbb{E}^{\log}_{n\leq X}  \lambda(\lfloor \alpha_1 n+\beta_1\rfloor)\lambda(\lfloor \alpha_2 n+\beta_2\rfloor)
\end{align*}
exists, and is $0$ if and only if for all large enough $m \in \mathbb{N}$ we have \[\lfloor \alpha_1 m+\beta_1\rfloor\neq r\lfloor \alpha_2 m+\beta_2\rfloor.\] 

\end{enumerate}
\end{theorem}

\textbf{Remarks.}
\begin{itemize}
    \item Note that Theorem~\ref{thm_twopoint} contains the statement that the logarithmic mean of $\lambda(\lfloor \alpha_1 n+\beta_1\rfloor)\lambda(\lfloor \alpha_2 n+\beta_2\rfloor)$ always exists. There are certain trivial examples when the mean value is non-zero (e.g. $\alpha_1 = \alpha_2 = 1$, $\beta_1 = \beta_2 = 0$), and some less trivial examples, e.g. $\alpha_1=\sqrt{2}$, $\alpha_2=2\sqrt{2}$, $\beta_1=0$, $\beta_2=1/4$.

    \item The case of Theorem~\ref{thm_twopoint}(2) where $\beta_i/\alpha_i$ are integers follows as a special case from a result of Frantzikinakis~\cite{Fr20}.
\end{itemize}

A tool for proving Theorem~\ref{thm_twopoint} is an analogue of the two-point logarithmic Elliott conjecture (proved by Tao in~\cite{tao}) where the summation variable is restricted to lie in a Bohr set. For ease of future reference we give the definition of these sets here. 

\begin{definition} Let $d\geq 1$, $\gamma\in \mathbb{R}^d$, and let $U\subset \mathbb{R}^d/\mathbb{Z}^d$ be measurable. Then we call
\begin{align*}
B_d(\gamma,U):=\{x\in \mathbb{Z}:\,\, \gamma x\in U \,\m \mathbb{Z}^d\}    
\end{align*}
an \emph{inhomogeneous Bohr set}.
\end{definition}
\noindent Viewing $[0,1)^d$ as a fundamental domain for $\mathbb{R}^d/\mathbb{Z}^d$, we denote
\begin{align*}
\mathcal{B}_{d,\textnormal{convex}}:=\{B_d(\gamma, U):\,\, \gamma \in \mathbb{R}^d, \, U \subset [0,1)^d, \, U\textnormal{ convex}\}.    
\end{align*}
Write $\mathcal{B}_{\textnormal{convex}}$ for $\bigcup_{d \geqslant 1} B_{d, \textnormal{convex}}$, and for $B \in \mathcal{B}_{ \textnormal{convex}}$ \[ \delta_B := \lim_{X \rightarrow \infty} \E_{n \leqslant X} 1_B(n) = \lim_{X \rightarrow \infty} \E_{n \leqslant X}^{\log} 1_B(n).\] It is a standard result (and follows from Lemma~\ref{le_ratdepend} below, for example) that the natural average $\delta_B$ is well-defined for all $B \in \mathcal{B}_{\textnormal{convex}}$. The equality of logarithmic and natural averages follows from partial summation. 

For stating the next theorem, we also need the notion of pretentious multiplicative functions, introduced in \cite{GS}.

\begin{definition} Let $f:\mathbb{N}\to [-1,1]$ be multiplicative. We say that $f$ is \emph{pretentious} if for some Dirichlet character $\chi$ we have \[ \sum_p \frac{1 - \Re(f(p) \overline{\chi}(p))}{p} < \infty.\]  Otherwise, we say that $f$ is \emph{non-pretentious}.
\end{definition}

The Liouville function is clearly non-pretentious by the prime number theorem in arithmetic progressions.

\begin{theorem}[Logarithmic two-point Elliott over Bohr sets]\label{thm_bohrchowla} Let $f_1,f_2:\mathbb{N}\to [-1,1]$ be multiplicative functions with $f_1$ non-pretentious. Let $B \in \mathcal{B}_{\textnormal{convex}}$. Then, for any $a_1,a_2\in \mathbb{N}$ and $h_1,h_2\in \mathbb{Z}$ satisfying $a_1h_2\neq a_2h_1$, we have
\begin{align*}
\lim_{X\to \infty} \mathbb{E}_{n\leq X}^{\log}f_1(a_1n+h_1)f_2(a_2n+h_2)1_{B}(n)=0.   
\end{align*}
\end{theorem}
\noindent We note that the case where $a_1=a_2=1$ and $B=B_d(\gamma,U)$ with $d=1$, and $U$ an interval essentially follows from~\cite{CHRST}. Indeed our methods are broadly similar to those from the excellent paper~\cite{CHRST} (though we were working independently from those authors). A few additional technical results are needed to prove Theorem~\ref{thm_bohrchowla}, to handle the rational dependencies that can arise when $d \geqslant 2$.

When $k \geqslant 3$, we have the following ``99\% version'' of Conjecture~\ref{conj_frantz}.

\begin{theorem}[99\% result for $k$-point correlations]
\label{thm_hom99}
Let $k \geqslant 3$ be an integer, and let $(\alpha_1,\alpha_2,\dots,\alpha_k):= \alpha \in \mathbb{R}_{>0}^k \setminus \mathbb{Q}^k$. \begin{enumerate}
    \item Suppose that $1,\alpha_1,\ldots, \alpha_k$ are linearly independent over $\mathbb{Q}$. Then there is some $\eta >0$ (depending on the $\alpha_i$'s) such that for any multiplicative functions $f_1,\ldots f_k:\mathbb{N}\to [-1,1]$ we have \begin{equation}
\label{eq_zerobetaconv}
\limsup_{X\to \infty} \Big\vert \E_{n\leqslant X}^{\log} \prod\limits_{i=1}^k f_i(\lfloor \alpha _i n \rfloor)-\prod_{i=1}^k \E_{n\leqslant X}^{\log} f_i(n)\Big\vert \leqslant 1 - \eta.
 \end{equation}
 \item Suppose that $\mathcal{V}$ is a nonempty maximal linearly independent set of vectors $v \in \mathbb{Z}^k$ for which $v \cdot \alpha \in \mathbb{Z}$ for all $v\in \mathcal{V}$. Suppose also that there exists a vector $(w_1,\dots,w_k) := w \in \mathbb{R}_{> 0}^k$ such that:
\begin{itemize}
\item $v \cdot w = 0$ for all $v\in \mathcal{V}$;
\item $w_1$ is the unique maximal coefficient of $w$.
\end{itemize}
Then there is some $\eta >0$ (depending on the $\alpha_i$'s) such that for any multiplicative non-pretentious function $f_1:\mathbb{N}\to [-1,1]$ and completely multiplicative functions $f_2,\dots,f_k: \mathbb{N} \to [-1,1]$ we have~\eqref{eq_zerobetaconv}. In particular, we have 
 \begin{equation*}
\limsup_{X\to \infty} \Big\vert \E_{n\leqslant X}^{\log} \prod\limits_{i=1}^k \lambda(\lfloor \alpha _i n \rfloor)\Big\vert \leqslant 1 - \eta.
 \end{equation*}
\end{enumerate}
\end{theorem}
\noindent We stress that in Theorem \ref{thm_hom99}(2) the first condition is indeed $v \cdot w = 0$ as an element of $\mathbb{R}$, and is not a shorthand for $v \cdot w \in \mathbb{Z}$ (as is sometimes the convention).

Theorem \ref{thm_hom99}(1) deals with the case when $1,\alpha_1,\dots,\alpha_k$ are linearly independent over $\mathbb{Q}$. At the opposite extreme, when the $\alpha_i$'s are as rationally dependent as possible, we can also show some cancellation. 

\begin{corollary}
\label{thm_ratmult}
Let $k \geqslant 3$, and let $\alpha_1,\dots,\alpha_k >0$ be distinct with $\max(\alpha_1,\dots,\alpha_k) = \alpha_1$. Suppose that there is some irrational $\beta$ such that $\alpha_i/\beta \in \mathbb{Q}$ for all $i$. Then there is some $\eta >0$ (depending on the $\alpha_j$'s) such that, for any multiplicative functions $f_1,\ldots f_k:\mathbb{N}\to [-1,1]$ with $f_1$ non-pretentious and $f_2,\dots,f_k$ completely multiplicative, we have~\eqref{eq_zerobetaconv}.
\end{corollary}

\begin{proof}
Write $\alpha = (\alpha_1,\dots,\alpha_k)$, and for $i$ in the range $1\leq i\leq k$ let $\alpha_i=q_i\beta$ (for some $q_i\in \mathbb{Q}_{>0}$). The $q_i$ are distinct. Now apply Theorem~\ref{thm_hom99}(2), taking $w=(q_1,\ldots, q_k)$. This is an admissible choice, since $v\cdot \alpha\in \mathbb{Z}$ for $v\in \mathcal{V}$ implies $v\cdot (q_1,\ldots, q_k)=0$. 
\end{proof}
For example, when $k=4$ we have results for tuples $(\alpha_1,\alpha_2,\alpha_3,\alpha_4)$ such as 
\begin{itemize}
\item $(\sqrt{2}, \sqrt{3}, \sqrt{5}, \sqrt{7})$ (rationally independent);
\item  $(\sqrt{2}, \sqrt{2} + \sqrt{3}, \sqrt{2} + 2\sqrt{3}, \sqrt{2} + 3 \sqrt{3})$ (take $\mathcal{V} = \{(1,-2,1,0), (0,1,-2,1)\}$ and $w = (1,2,3,4)$, say); and 
\item $(\sqrt{2}, 2\sqrt{2}, 3 \sqrt{2}, 4 \sqrt{2})$ (take $w = (1,2,3,4)$ again). 
\end{itemize} But our methods cannot handle the tuple $(\alpha_1,\alpha_2,\alpha_3,\alpha_4) = (\sqrt{2}, \sqrt{2}+1, \sqrt{3}, \sqrt{3}+1)$, at least not without the injection of some further ideas.

Theorem~\ref{thm_hom99} is proved by a rather simple argument. After handling the case of pretentious $f_i$ by almost periodicity of such functions, we restrict $n$ to a suitably chosen Bohr set and then replace $n$ by a multiple $rn$ that reduces the $k$-point correlation to a $2$-point correlation. From this, Theorem~\ref{thm_twopoint} can be applied. 

The main challenge is establishing that the Bohr set is non-empty, and this leads to the various conditions in Theorem \ref{thm_hom99}(2). The requirement that the functions $f_2,\dots,f_k$ are completely multiplicative (rather than merely multiplicative) can be relaxed to the assumption that $f_2,\dots,f_k$ are completely multiplicative at a single common prime. However we have not been able to prove Theorem \ref{thm_hom99}(2) for functions that are only assumed to be multiplicative. 

We also prove the following extension of the ``99\% Elliott conjecture'' due to the first author~\cite{teravainen}.

\begin{theorem}[99\% Elliott over Bohr sets]
\label{thm_highercorrelationsandBohr}
Let $k \geqslant 3$, and let $a_1,\dots,a_k \in \mathbb{N}$ and $h_1,\dots,h_k\in \mathbb{Z}$ with $a_i h_j - a_jh_i \neq 0$ for all $i \neq j$. Let $B \in \mathcal{B}_{\textnormal{convex}}$. Then there is some $\eta >0$ for which the following holds. For any multiplicative functions $f_1,f_2,\dots f_k:\mathbb{N}\to [-1,1]$ with $f_1$ non-pretentious, \[ \limsup_{X\to \infty}\Big\vert\E_{n \leqslant X}^{\log} 1_{B}(n)\prod\limits_{i=1}^k f_i(a_in + h_i) \Big\vert \leqslant \delta_B(1 - \eta).\] 
\end{theorem}
\noindent This result is not needed in the proof of Theorem~\ref{thm_hom99}, however. 

\subsection{Acknowledgements}
The majority of the work for this note was done in the first half of 2021, partly when both authors were Junior Fellows at the Number Theory programme at Institut Mittag-Leffler (working remotely). JT was supported by a Titchmarsh Fellowship, Academy of Finland grant no. 340098, a von Neumann Fellowship (NSF grant \texttt{DMS-1926686}),  and funding from European Union's Horizon
Europe research and innovation programme under Marie Sk\l{}odowska-Curie grant agreement No
101058904. AW was supported by a Junior Research Fellowship at Trinity College Cambridge. 

We thank Nikos Frantzikinakis for helpful comments.

\section{Notation and some preliminaries}

As usual, we denote $e(\theta):=e^{2\pi i\theta}$. We use standard Landau and Vinogradov asymptotic notation $O(\cdot), o(\cdot), \ll, \gg$. To clarify a couple of points, a function denoted by $o_c(1)$ will tend to zero as $X \rightarrow \infty$ with the parameter $c$ fixed. A function denoted by $o_{P \rightarrow \infty}(1)$ is a function that tends to zero as $P \rightarrow \infty$ (with all other parameters fixed). 

We say that a sequence $(a(n))_{n\in \mathbb{N}}$ taking values in a $d$-dimensional torus $T$ is \emph{equidistributed} if 
\begin{align}\label{eq_equidistribute}
 \lim_{X\to \infty}\frac{1}{X}\sum_{n\leq X}F(a(n))=\int_{T} F\d\mu,  
\end{align}
for all continuous functions $F:T\to \mathbb{C}$, where $\mu$ is the Haar measure on $T$. We say that $(a(n))_{n\in \mathbb{N}}$ is \emph{totally equidistributed} if $(a(qn+b))_{n\in \mathbb{N}}$ is equidistributed for all $q,b\in \mathbb{N}$. It is well known (see~\cite[Proposition 1.1.2]{tao12}) that~\eqref{eq_equidistribute} is equivalent to the same statement holding for all $f$ of the form $1_{U}$, where $U\subset T$ is an open set whose boundary has measure zero.  

We shall frequently use (sometimes without further mention) the Kronecker--Weyl theorem, which states that  for $\alpha \in \mathbb{R}^d/\mathbb{Z}^d$ the sequence $(\alpha n)_{n\in \mathbb{N}}$ equidistributes in the torus $\mathbb{R}^d/\mathbb{Z}^d$ if and only if $k\cdot \alpha\not \in \mathbb{Z}$ for all $k\in \mathbb{Z}^d$.

We endow $\mathbb{R}^d/\mathbb{Z}^d$ with the usual metric $\Vert x - y\Vert_{\mathbb{R}^d/\mathbb{Z}^d} = \min_{z \in \mathbb{Z}^d} \vert x - y- z\vert$. A function $F: \mathbb{R}^d/\mathbb{Z}^d \longrightarrow \mathbb{C}$ is Lipschitz, with Lipschitz constant $c\in \mathbb{R}_{ \geqslant 0}$, if $c = \sup_{\substack{ x,y \in \mathbb{R}^d/\mathbb{Z}^d \\ x \neq y}} \frac{ \vert F(x) - F(y)\vert}{\Vert x - y\Vert_{\mathbb{R}^d/\mathbb{Z}^d}}.$

\section{Decomposition of Bohr sets}

The goal of this section is to prove Lemma~\ref{le_approximation}, a result on Fourier approximations of Bohr sets in $\mathcal{B}_{\textnormal{convex}}$. Such a result is surely standard, but we couldn't find exactly the statement we needed in an easily citable form. 

We begin with a lemma to deal with possible rational dependencies between the coordinates of the phase. 
\begin{lemma}[Removing rational dependencies]
\label{le_ratdepend}
Let $d \geqslant 1$, and let $B_d(\gamma,U)$ be an inhomogeneous Bohr set with $\gamma \notin \mathbb{Q}^d$. Then there is an integer $d^\prime$ in the range $1 \leqslant d^\prime \leqslant d$, a vector $(\rho_1,\dots,\rho_{d^\prime})^T = \rho \in \mathbb{R}^{d^{\prime}}$ for which $1,\rho_1,\dots \rho_{d^\prime}$ are linearly independent over $\mathbb{Q}$, an integer $q \geqslant 1$, and measurable sets $U^\prime(1),\ldots, U^\prime (q) \subset [0,1)^{d^\prime}$ for which \[ 1_{B_d(\gamma,U)}(n) = 1_{B_{d^\prime}(\rho, U^\prime(n \Mod q))}(n).\] Furthermore there is a constant $C(\gamma)$ such that, if $U \subset [0,1)^d$ is convex, each set $U^\prime(a)$ is a disjoint union of at most $C(\gamma)$ convex sets. Finally, \[ \frac{1}{q} \sum\limits_{a \leqslant q} \vol(U^\prime(a)) = \delta_{B_d(\gamma,U)}.\] 
\end{lemma}
\begin{proof}
By the abelian Ratner's theorem of~\cite[Proposition 1.1.5]{tao12} we may write $\gamma = \gamma^\prime + \gamma^{\prime\prime}$ where $\gamma^{\prime\prime} \in \mathbb{Q}^d$ and $\gamma^\prime n$ mod $\mathbb{Z}^d$ totally equidistributes in some subtorus $T \leqslant \mathbb{R}^d/\mathbb{Z}^d$. Let $d^\prime := \dim T$, noting that $d^\prime \geqslant 1$ (since $\gamma \notin \mathbb{Q}^d$ by assumption). 

Let $q \in \mathbb{N}$ be minimal such that $q \gamma^{\prime\prime} \in \mathbb{Z}^d$. Define $U_1(n)$ to be the representative of $(U - n \gamma^{\prime\prime}) \cap T$ mod $\mathbb{Z}^d$ in the fundamental domain $[0,1)^d$. Observe also that $n\gamma \in U$ mod $\mathbb{Z}^d$ if and only if $n \gamma^\prime \in (U - n \gamma^{\prime\prime}) \cap T$ mod $\mathbb{Z}^d$. Since $U_1(n)$ depends only on $n$ mod $q$, \[1_{B_d(\gamma,U)}(n)= 1_{B_{d}(\gamma^\prime, U_1(n \, (\text{mod } q)))}(n).\]
There is a linear transformation $M \in SL_d(\mathbb{Z})$ (which has a well-defined action on $\mathbb{R}^d/\mathbb{Z}^d$) such that $M(T) = (\mathbb{R}^{d^\prime}/\mathbb{Z}^{d^\prime}) \times \{0\}^{d - d^\prime}$. Let $U^\prime(n): = M(U_1(n)) \, \text{mod } \mathbb{Z}^d$ (with the $d-d^\prime$ trailing zeros removed and viewed as a subset of $[0,1)^{d^\prime}$). Let $\rho = M(\gamma^\prime)$, and again remove the final $d-d^\prime$ coordinates (which are all integers) to view $\rho \in \mathbb{R}^{d^\prime}$. Since $\rho n$ mod $\mathbb{Z}^{d^\prime}$ totally equidistributes in $\mathbb{R}^{d^\prime}/\mathbb{Z}^{d^\prime}$ by construction, we conclude from the Kronecker--Weyl theorem that $1,\rho_1,\dots \rho_{d^\prime}$ are linearly independent over $\mathbb{Q}$. As $B_d(\gamma^\prime,U_1(n \, (\text{mod } q)))=  B_{d^\prime}(\rho, U^\prime(n \, (\text{mod } q)))$, the first part of the lemma follows. 

For the second part of the lemma, note that $T \subset [0,1)^d$ is a disjoint union of finitely many convex sets (each a translation of a fixed linear subspace intersected with $[0,1)^d$). Therefore, if $U \subset [0,1)^d$ is convex, $U_1(n)$ is a disjoint union of finitely many convex sets. Hence $M(U_1(n)) \subset \mathbb{R}^{d^\prime} \times \mathbb{Z}^{d -d^\prime}$ is also a union of disjoint convex sets, say $M(U_1(n)) = \bigcup_{k \leqslant K} S_k$. Reducing modulo $\mathbb{Z}^{d^\prime}$ to give $U^\prime(n) \subset [0,1)^{d^\prime}$ may split each convex set $S_k$ into a union of possibly $2^{d^\prime}$ convex sets, but this larger collection still remains disjoint, as the points in $M(U_1(n))$ are distinct modulo $\mathbb{Z}^{d}$. 
\end{proof}

We now formulate the following result for approximating Bohr sets by trigonometric polynomials. 

\begin{lemma}[Approximation of Bohr sets by trigonometric polynomials and periodic part]\label{le_approximation}
Let $d\geq 1$ and $\alpha\in \mathbb{R}^d$ be fixed. Let $B = B_d(\alpha,U)\in \mathcal{B}_{\textnormal{convex}}$. Then there exists an integer $q \geqslant 1$ (depending only on $\alpha$) and for every $\varepsilon >0$ a decomposition of functions
\begin{align*}
1_{B}(n)=T_{\varepsilon}(n)+ \sum\limits_{a \leqslant q} t_a 1_{n \equiv a \, \textnormal{mod } q} + \mathcal{E}_{\varepsilon}(n)    
\end{align*}
such that the following hold.
\begin{enumerate}[(i)]\item For some constant $K_{\varepsilon} \ll_{\varepsilon} 1$, some sequence of real numbers $(\gamma_{k,\varepsilon})_{k\geq 1}$, and some complex numbers $c_{\varepsilon}(k)$ with $|c_{\varepsilon}(k)|\ll_{\varepsilon} 1$  we have  
\begin{align*}
T_{\varepsilon}(x)=\sum_{1\leq k\leq K_{\varepsilon}} c_{\varepsilon}(k)e(\gamma_{k,\varepsilon}x) \end{align*} for all $x \in \mathbb{R}$. Furthermore, if $\alpha \notin \mathbb{Q}^d$ then $\gamma_{k,\varepsilon} \notin \mathbb{Q}$ for all $k$.
\item We have $t_a \geqslant 0$ for all $a$ and $\frac{1}{q} \sum_{a \leqslant q} t_a = \delta_B + O(\varepsilon)$. 
\item $\limsup_{X\to \infty}\mathbb{E}_{n\leq X}|\mathcal{E}_{\varepsilon}(n)|\leq \varepsilon$.   
\end{enumerate}
\end{lemma}
\begin{corollary}[Approximation of Bohr sets by trigonometric polynomials]
\label{cor_approx}
Let $B = B_d(\alpha,U)\in \mathcal{B}_{\textnormal{convex}}$. Then for every $\varepsilon >0$ there exists a  decomposition \[ 1_B(n) = T_{\varepsilon}(n) + \mathcal{E}_{\varepsilon}(n)\] with $T_{\varepsilon}$ and $\mathcal{E}_{\varepsilon}$ having the same properties as in the conclusion of Lemma~\ref{le_approximation}, save for the fact that some of the phases $\gamma_{k,\varepsilon}$ may be rational. 
\end{corollary}
\begin{proof}[Proof of Corollary~\ref{cor_approx}]
Expand $\sum_{a \leqslant q} t_a 1_{n \equiv a \, \text{mod } q} = \frac{1}{q} \sum_{a,r \leqslant q} t_a e\Big( \frac{-ra}{q}\Big) e\Big( \frac{rn}{q}\Big)$ and amalgamate with the original trigonometric polynomial $T_{\varepsilon}$. The coefficients $c_{\varepsilon}(k)$ remain suitably bounded, since $\vert  \frac{1}{q}\sum_{a \leqslant q} t_a e(\frac{-ra}{q})\vert \leqslant \frac{1}{q}\sum_{a \leqslant q} \vert t_a\vert \leqslant 1 + O(\varepsilon) = O_{\varepsilon}(1).$
\end{proof}

\begin{proof}[Proof of Lemma~\ref{le_approximation}] If $\alpha\in \mathbb{Q}^d$ then $1_B(n)$ is periodic so may be written exactly as $\sum_{a \leqslant q} t_a 1_{n \equiv a \, \text{mod } q}$ (for some $q$), with no error. Each $t_a \geqslant 0$, and $\frac{1}{q} \sum_{a \leqslant q} t_a = \delta_B$ exactly. 

If $\alpha \notin \mathbb{Q}^d$, we use Lemma~\ref{le_ratdepend} to construct $d^\prime$, $q$, $\rho \in \mathbb{R}^{d^\prime}$, and sets $U^\prime(1), \dots, U^\prime(q) \subset [0,1)^{d^\prime}$; expanding the condition $n \equiv a \, (\text{mod } q)$ in additive characters, we get \[ 1_B(n) = 1_{B_{d^\prime}( \rho, U^\prime(n \, (\text{mod } q)))}(n) = \frac{1}{q}\sum\limits_{a,r=1}^q e\Big(-\frac{ra}{q}\Big) 1_{B_{d^\prime}(\rho, U^\prime(a))}(n) e\Big(\frac{r}{q}n\Big).\] From the second part of Lemma~\ref{le_ratdepend}, write $U^\prime(a)$ as union $\bigcup_{l \leqslant L} S_{a,l}$ of disjoint convex sets $S_{a,l} \subset [0,1)^{d^\prime}$. By further subdivision as necessary, we may assume that each $S_{a,l}$ is contained in a Cartesian box of side-length $\frac{1}{10}$. Note that $L$ depends only on $\alpha$. 

By~\cite[Corollary A.3]{GT_linear}, we can write 
\begin{equation}
\label{eq:approxofconvex}
1_{S_{a,l}} = F_{\varepsilon,S_{a,l}} + O(G_{\varepsilon,S_{a,l}}),
\end{equation} where $F_{\varepsilon,S_{a,l}}, G_{\varepsilon,S_{a,l}}: \mathbb{R}^{d^\prime}\longrightarrow [0,1]$ are non-negative Lipschitz functions with Lipschitz constants $O(\varepsilon^{-1})$, where both functions are supported within Cartesian boxes of side-length $\frac{1}{5}$, and where $\int_{\mathbb{R}^{d^\prime}} G_{\varepsilon,S_{a,l}}(x) \d x = O(\varepsilon)$. Because of their restricted support, we may consider $F_{\varepsilon,S_{a,l}}, G_{\varepsilon,S_{a,l}}$ as Lipschitz functions on $\mathbb{R}^{d^\prime}/ \mathbb{Z}^{d^\prime}$ with Lipschitz constant $O(\varepsilon^{-1})$, and furthermore where $\int_{\mathbb{R}^{d^\prime}/\mathbb{Z}^{d^\prime}} G_{\varepsilon,S_{a,l}}(x)  \d x = O(\varepsilon)$.

From~\cite[Lemma A.9]{GT_quadratic}, we obtain  (for all $K$ sufficiently large)
\begin{align}
\label{eq_trigdecomp}
1_{B_{d^\prime}(\rho,U^\prime(a))}(n) &= \sum\limits_{l \leqslant L}(F_{\varepsilon,S_{a,l}}(\rho n) + O(G_{\varepsilon,S_{a,l}}(\rho n))) \nonumber\\
&= \sum\limits_{l \leqslant L}\Big(\sum_{\substack{k \in \mathbb{Z}^{d^\prime} \\ \Vert k\Vert_{\infty} \leqslant K}} c_{K,\varepsilon, a,l}(k) e(n k \cdot \rho) + O\Big(\frac{\log K}{\varepsilon K}\Big) + O(G_{\varepsilon,S_{a,l}}(\rho n))\Big)
\end{align}
for some complex coefficients $c_{K,\varepsilon, a,l}(k)$ with $\vert c_{K,\varepsilon,a,l}(k)\vert \ll_{\varepsilon} 1$. Choose $K = K_{\varepsilon}$ sufficiently large so that $(\log K) \varepsilon^{-1}K^{-1} \leqslant \varepsilon$.  Note that (as the sequence $\rho n$ equidistributes in $\mathbb{R}^{d^\prime}/\mathbb{Z}^{d^\prime}$) we have $\lim_{X \rightarrow \infty} \E_{n \leqslant X}  G_{\varepsilon,S_{a,l}}(\rho n) = \int_{\mathbb{R}^{d^\prime}/\mathbb{Z}^{d^\prime}} G_{\varepsilon, S_{a,l}}(x) \d x = O(\varepsilon)$. Therefore, inserting the sums over $a$, $r$ into~\eqref{eq_trigdecomp} and separating out the $k=0$ term, we get 
\begin{align}
\label{eq_threeterm}
1_{B}(n) = \sum\limits_{\substack{k \in \mathbb{Z}^{d^\prime} \\  \Vert k\Vert_{\infty} \leqslant K \\ k \neq 0}} \sum\limits_{r \leqslant q}e(n(k \cdot \rho + \frac{r}{q}))&\Big( \frac{1}{q}\sum\limits_{l \leqslant L} \sum\limits_{a\leqslant q}e\Big( - \frac{ra}{q}\Big) c_{K, \varepsilon, a,l}(k)\Big) \nonumber\\ &+ \sum_{a \leqslant q} 1_{n \equiv a \, \text{mod } q} \sum\limits_{l \leqslant L} c_{K,\varepsilon, a,l}(0) + \mathcal{E}_{\varepsilon}(n) 
\end{align} where $\limsup\limits_{X \rightarrow \infty} \E_{n \leqslant X} \vert \mathcal{E}_{\varepsilon}(n)\vert = O(\varepsilon)$.

When $k \in \mathbb{Z}^{d^\prime} \setminus \{0\}$, Lemma~\ref{le_ratdepend} ensures that $k \cdot \rho + \frac{r}{q} \notin \mathbb{Q}$. Therefore, replacing $\varepsilon$ by $\varepsilon/b_{\alpha}$ for a suitable constant $b_{\alpha}$, the first term satisfies the conditions to be $T_{\varepsilon}(n)$ and $\mathcal{E}_{\varepsilon}$ is a suitable error. It remains to prove part (ii) of the lemma. 

By summing~\eqref{eq_threeterm}  over $n \leqslant X$ (and using the fact that $\sum_{n \leqslant X} e(n(k \cdot \rho + \frac{r}{q}) )= O(1)$ uniformly in $X$)\[ \E_{n \leqslant X} 1_B(n) = \frac{1}{q}\sum\limits_{a \leqslant q} \sum\limits_{l \leqslant L} c_{K,\varepsilon,a,l}(0) + O(\varepsilon)\] for large enough $X$. From the construction of the $c_{K,\varepsilon,a,l}(k)$ in~\cite[Lemma A.9]{GT_quadratic}, we also derive \[c_{K,\varepsilon, a,l}(0) = \int_{\mathbb{R}^{d^\prime}/\mathbb{Z}^{d^\prime}} F_{\varepsilon, S_{a,l}}(x) \d x \geqslant 0.\] Setting $t_a = \sum_{l \leqslant L} c_{K,\varepsilon,a,l}(0)$, the part (ii) of the lemma follows.
\end{proof}

\section{Lemmas on correlations}

\subsection{Correlations twisted by additive characters}

In this section, we prove a correlation estimate for multiplicative functions twisted by linear phases (Lemma~\ref{le_frantzikinakis}) that is important in the proof of our main theorems. We also resolve the pretentious case of the proofs of our main theorems in Lemma~\ref{le_pretentious}. We begin by summarising some known correlation estimates of Tao~\cite{tao}, the first author~\cite{teravainen}, and Frantzikinakis--Host~\cite{fh-IMRN}.

\begin{lemma}\label{le_tao_fh}
Let $k\geq 1$, and let $a_1,\ldots, a_k>0$ and $h_1,\ldots, h_k\in \mathbb{N}$ be integers with $a_ih_j-a_jh_i\neq 0$ for all $i\neq j$. Let $f_1,\ldots, f_k:\mathbb{N}\to [-1,1]$ be multiplicative functions.
\begin{enumerate}
    \item Suppose that $f_1$ is non-pretentious. Then we have
   \begin{align*}
\lim_{X\to \infty}\mathbb{E}_{n\leq X}^{\log}f_1(a_1n+h_1)f_2(a_2n+h_2)=0.
\end{align*}

\item Suppose that $f_1$ is non-pretentious. Then for some $\eta>0$, depending only on the values $a_i,h_i$, we have
\begin{align*}
 \limsup_{X\to \infty}|\mathbb{E}_{n\leq X}^{\log}\prod_{j=1}^k f_j(a_jn+h_j)|\leq 1-\eta.   
\end{align*}

\item For any irrational $\gamma\in \mathbb{R}$ we have
   \begin{align*}
\lim_{X\to \infty}\mathbb{E}_{n\leq X}^{\log}e(\gamma n)\prod_{j=1}^k f_j(n+h_j)=0.
\end{align*}
\end{enumerate}
\end{lemma}

\begin{proof}
Part (1) follows from Tao's resolution of the two-point logarithmic Elliott conjecture~\cite[Theorem 1.3]{tao}, after noting that the non-pretentiousness assumption on $f_1$ there (which involves archimedean characters $n^{it}$) can be weakened in the case of real-valued functions $f_1$ using~\cite[Lemma C.1]{MRT}. Part (3) is the ``irrational logarithmic Elliott conjecture'' of Frantzikinakis--Host~\cite[Corollary 1.4]{fh-IMRN}.

It remains to prove part (2). If we assume that $f_1$ takes values in $\{-1,+1\}$, then part (2) follows immediately from the ``99\% Elliott conjecture'' of the first author ~\cite[Theorem 2.6]{teravainen} (using partial summation to pass to the logarithmic average). To deal with the general case\footnote{Alternatively, one could adapt the methods from \cite{teravainen}. Indeed, \cite[Proposition 5.4]{teravainen} as stated is for multiplicative functions taking values which are $q^{th}$ roots of unity for some fixed $q$. It is easy to adapt the proof to the case of multiplicative functions taking values in the convex hull of the $q^{th}$ roots of unity, which when $q=2$ gives the full interval $[-1,+1]$.} when $f_1$ takes values in $[-1,+1]$, we use an argument of Tao \cite[Proposition 2.1]{tao}. Write $f_1 = f_1^\prime f_1^{\prime\prime}$, where $f_1^\prime(n) = \vert f_1(n) \vert$ and $f_1^{\prime\prime}(n)= \sgn(f_1)$. Let $A$ be a sufficiently large quantity (depending on the $a_i$, $h_i$, and the value of $\eta$ that can be established in part (2) when $\vert f_1(n)\vert = 1$ for all $n$). We may assume that \[ \sum\limits_{p} \frac{1 - f_1^\prime(p)}{p} < A. \] Indeed, if not then using the standard elementary bound \[ \E_{n \leqslant X}^{\log} f^\prime(n) \ll \exp(-\sum\limits_{p \leqslant X} \frac{1 - f^\prime(p)}{p}),\] which holds for any non-negative multiplicative function, we conclude that \[ \limsup_{X \rightarrow \infty}\E_{n \leqslant X}^{\log} f_1^\prime(n) = o_{A \rightarrow \infty}(1).\] Using non-negativity again we derive \[ \limsup_{X \rightarrow \infty}\E_{n \leqslant X}^{\log} f_1^\prime(a_1n + b_1) = o_{A \rightarrow \infty}(1),\] and so by the triangle inequality we may conclude that \[\limsup_{X \rightarrow \infty}\vert \E_{n \leqslant X}^{\log} \prod\limits_{i=1}^k f_i(a_in + h_i)\vert \leqslant 1- \eta\] as required. 

Now, for later purposes we let $S$ be the set of $\{-1,+1\}$-valued multiplicative functions $g$ for which \[ \sum_{p} \frac{1 - g(p)}{p} < A^2.\] We also construct a random multiplicative function $\mathbf{f_1^\prime} $ taking values in $\{-1,+1\}$ by taking $\mathbf{f_1^\prime}(p^j)$ to be independent $\{-1,+1\}$-valued random variables with mean $\E \mathbf{f_1^\prime}(p^j) = f_1^\prime(p^j)$. (There is a slight overloading of the symbol $\E$ in what follows, but we hope that it will be clear that $\E_{n \leqslant X}^{\log}$ refers to logarithmic averaging and $\E$ refers to expectation of a random variable.) By Fubini's theorem we have \[ \E \sum_{p} \frac{1 - \mathbf{f_1^\prime}(p)}{p} < A,\] so by Markov's inequality we have $\mathbf{f_1^\prime} \in S$ with probability at least $1 - O(A^{-1})$. Supposing that $\mathbf{f_1^\prime} \in S$, set $\mathbf{f_1}: = \mathbf{f_1^\prime} f_1^{\prime\prime}$. Thus $\mathbf{f_1}$ is a random multiplicative function taking values in $\{-1,+1\}$ such that $\E \mathbf{f_1}(n) = f_1(n)$ for all $n$. By the triangle inequality we have 
\begin{align*}
 \vert \mathbf{f_1}(p) - f_1(p) \vert & = \vert f_1^{\prime\prime}(p)( \mathbf{f_1^\prime}(p) - f_1^\prime(p)) \vert \\
&= \vert f_1^{\prime\prime}(p) ((1 - f_1^\prime(p)) - (1 - \mathbf{f_1^\prime}(p)))\vert \\
& \leqslant (1 - f_1^\prime(p)) + (1 - \mathbf{f_1^\prime}(p)).
\end{align*}
In particular \[ \sum_p \frac{ \mathbf{f_1}(p) \overline{\chi}(p)}{p} = \sum_p \frac{ f_1(p) \overline{\chi}(p)}{p}  + O_A(1).\] Taking real parts, since $f_1$ is non-pretentious we conclude that $\mathbf{f_1}$ is non-pretentious. Since $\mathbf{f_1}$ takes values in $\{-1,+1\}$, by \cite[Theorem 2.6]{teravainen} we get 
\begin{align}\label{eq:compmultcase}
\limsup_{X\to \infty}\vert \E_{n\leqslant X}^{\log} \mathbf{f_1}(a_1n + h_1) \prod\limits_{i=2}^k f_i(a_in + h_i)\vert \leqslant 1 - \eta
\end{align}
for some absolute constant $\eta>0$ (depending on $a_i,h_i$ but not on any of the multiplicative functions).

Therefore, by~\eqref{eq:compmultcase} and the reverse Fatou's lemma, for some $v\in \{-1,+1\}$ we have
\begin{align*}
\limsup_{X\to \infty}\vert \E_{n \leqslant X}^{\log} \prod\limits_{i=1}^k f_i(a_in + h_i) \vert 
=&\limsup_{X\to \infty}\E_{n \leqslant X}^{\log} v\prod\limits_{i=1}^k f_i(a_in + h_i)\\ 
=&\limsup_{X\to \infty}\E_{n \leqslant X}^{\log} v\E \mathbf{f_1}(a_in + h_i) \prod\limits_{i=2}^k f_i(a_in + h_i) \\
=&\limsup_{X\to \infty} \E_{n \leqslant X}^{\log} v\E (1_{S}(\mathbf{f_1}) + 1_{S^c}(\mathbf{f_1}))\mathbf{f_1}(a_in + h_i) \prod\limits_{i=2}^k f_i(a_in + h_i)\\
\leqslant & \E 1_S(\mathbf{f_1})\limsup_{X\to \infty} \E_{n \leqslant X}^{\log}v\mathbf{f_1}(a_in + h_i) \prod\limits_{i=2}^k f_i(a_in + h_i)  + \E 1_{S^c}(\mathbf{f_1}) \\
\leqslant & 1 - \eta + O(A^{-1}) \\
\leqslant & 1 - \frac{\eta}{2}
\end{align*}
if $A$ is large enough. Thus, replacing $\eta$ by $\eta/2$ we see that part (2) holds for general non-pretentious multiplicative functions $f_1: \mathbb{N} \to [-1,1]$. 
\end{proof}

As we will soon see, Theorems~\ref{thm_bohrchowla} and~\ref{thm_highercorrelationsandBohr} follow quickly from Lemma~\ref{le_approximation} and the following estimate (which is based heavily on Lemma~\ref{le_tao_fh}).

\begin{lemma}\label{le_frantzikinakis} 
Let $k \geqslant 1$, and let $a_1,\ldots, a_k>0$ and $h_1,\ldots, h_k\in \mathbb{N}$ be integers with $a_ih_j-a_jh_i\neq 0$ for all $i\neq j$. Let $f_1,\ldots, f_k:\mathbb{N}\to [-1,1]$ be multiplicative functions.
\begin{enumerate}
\item Suppose that $f_1$ is non-pretentious. Then for all $\gamma\in \mathbb{R}$ we have
\begin{align*}
\lim_{X\to \infty}\mathbb{E}_{n\leq X}^{\log}f_1(a_1n+h_1)f_2(a_2n+h_2)e(\gamma n)=0. \end{align*}
\item Suppose that $f_1$ is non-pretentious. If $\gamma \in \mathbb{Q}$ there is some $\eta>0$ (depending only on $\gamma$, the $a_i$ and the $h_i$) such that 
\[ \limsup_{X\to \infty}\vert\E_{n \leqslant X}^{\log} e(\gamma n) \prod\limits_{i=1}^k f_i(a_in + h_i) \vert \leqslant 1 - \eta.\] 
\item If $\gamma \notin \mathbb{Q}$, then \begin{align}\label{eq_Elliott_twist}\lim_{X \rightarrow \infty} \E_{n \leqslant X}^{\log} e(\gamma n) \prod\limits_{i=1}^k f_i(a_in + h_i)   = 0.\end{align} 
\end{enumerate}
\end{lemma}

\begin{proof}\emph{Case 1: $\gamma$ rational.} Write $\gamma=a/b$ with $a\in \mathbb{Z}$ and $b\in \mathbb{N}$. Then by expanding $e(\gamma n)$ as a linear combination of indicators of arithmetic progressions modulo $b$, for part (1) it suffices to show that for each $1\leq r\leq b$ we have
\begin{align*}
 \mathbb{E}_{n\leq X}^{\log}f_1(a_1n+h_1)f_2(a_2n+h_2)1_{n\equiv r \, \text{mod }b}=o(1).   
\end{align*}
Making a change of variables, this reduces to
\begin{align*}
\mathbb{E}_{m\leq X/b}^{\log}f_1(a_1(bm+r)+h_1)f_2(a_2(bm+r)+h_2)=o(1).   
\end{align*}
But this follows from Lemma~\ref{le_tao_fh}(1).

For part (2) when $\gamma \in \mathbb{Q}$, proceeding analogously we seek some $\eta>0$ for which 
\[\limsup_{X\to \infty}\vert\E_{m \leqslant X/b} ^{\log} \prod\limits_{i=1}^k f_i(a_i bm + a_i r + h_i) \vert \leqslant 1- \eta\]
for each $1\leq r\leq b$.
This follows directly from Lemma~\ref{le_tao_fh}(2). 

\emph{Case 2: $\gamma$ irrational.} In this case, the same argument works for parts (1) and (3), so we write out the argument for general $k$. We first reduce to the case where $f_1,\ldots, f_k$ are completely multiplicative. For each $1\leq i\leq k$, write $f_i=\widetilde{f_i}*g_i$, where $\widetilde{f_i}$ is the completely multiplicative function given on the primes by $\widetilde{f_i}(p)=f_i(p)$, and $g_i$ is the multiplicative function given on prime powers $p^{\ell}$ ($\ell \geqslant 1$) by $g_i(p^\ell)=f_i(p^\ell)-f_i(p) f_i(p^{\ell-1})$. Note that $|g_i(p^\ell)|\leq 2$ for all $p,\ell$, and $g_i(p) = 0$. 

Writing $f_i(n)=\sum_{d\mid n}g_i(d)\widetilde{f_i}(n/d)$ and applying the triangle inequality,~\eqref{eq_Elliott_twist} reduces to showing that
\begin{align*}
\sum_{d_1,\ldots, d_k\geq 1} |g_1(d_1)|\cdots |g_k(d_k)|\left|\mathbb{E}_{n\leq X}^{\log}e(\gamma n)\prod_{i=1}^k \widetilde{f_i}\left(\frac{a_in+h_i}{d_i}\right)1_{d_i\mid a_in+h_i}\right|=o(1).   
\end{align*}
If the system of $k$ congruences $a_ix+b_i\equiv 0\pmod{d_i}$ with $1\leq i\leq k$ has a solution, then there is a unique solution of the form $x\equiv c\pmod D$, where $D$ is the least common multiple of $d_1,\dots,d_k$. Making the change of variables $n=Dm+c$ in \eqref{eq_comp_multiplicative}, for any $w\geq 1$ the contribution from the terms with $d_1>w$ is
\begin{align*}
 \ll \sum_{\substack{d_1,\ldots, d_k\geq 1\\d_1>w}}\frac{|g_1(d_1)|\cdots |g_k(d_k)|}{D}&\ll w^{-1/3} \sum_{d_1,\ldots, d_k\geq 1}\frac{|g_1(d_1)|\cdots |g_k(d_k)|}{D^{2/3}}\\
 &\ll w^{-1/3} \prod_p\Big(1+ \sum\limits_{\substack{(i_1,\ldots,i_k) \in \mathbb{Z}_{ \geqslant 0}^k \\ \max i_j \geqslant 1}} \frac{ \vert g_1(p^{i_1})\vert \cdots \vert g_k(p^{i_k})\vert}{(p^{\max i_j})^{2/3}}\Big)\\
 &\ll w^{-1/3}.   
\end{align*}
\noindent Similarly, the contribution of terms with $d_j>w$ for some $j$ is $\ll w^{-1/3}$. Letting $w\to \infty$, we see that it suffices to show that for any fixed $d_1,\ldots, d_k\geq 1$ we have 
\begin{align}\label{eq_comp_multiplicative}
\mathbb{E}_{n\leq X}^{\log}e(\gamma n)\prod_{i=1}^k \widetilde{f_i}\left(\frac{a_in+h_i}{d_i}\right)1_{d_i\mid a_in+b_i}=o(1).    
\end{align}
Substituting $n=Dm+c$ in \eqref{eq_comp_multiplicative}, we reduce to proving 
\begin{align*}
\mathbb{E}_{m\leq x/D}^{\log}e(\gamma D m)\prod_{i=1}^k \widetilde{f_i}\left(\frac{a_i(Dm+c)+h_i}{d_i}\right)=o(1).        
\end{align*}
The linear polynomials $a_i'x+h_i':=\frac{a_i(Dx+c)+h_i}{d_i}$ have integer coefficients by assumption, and we have $a_i'h_j\neq a_j'h_i'$ whenever $i\neq j$. Hence, the claim~\eqref{eq_Elliott_twist} would follow from the case of completely multiplicative functions.  

Thus, we assume that each $f_i$ is completely multiplicative and that $(a_i,h_i)=1$ for all $i\leq k$, since otherwise we can pull out the common factors by complete multiplicativity. We may further assume that $f_i(a_i)=1$ for all $i\leq k$, since the values of $f_i$ at the primes dividing $a_i$ do not influence~\eqref{eq_Elliott_twist}. 

Let $A = \prod_{i \leqslant k} a_i$, $h_i^\prime = h_i \prod_{j \neq i} a_j$. Then, writing $\gamma^\prime = \gamma/A$, by complete multiplicativity and the fact that $f_i(a_i)=1$ for all $i\leq k$, it suffices to show that  \[ \E_{ n \leqslant X}^{\log} e(\gamma^\prime A n)\prod\limits_{i=1}^kf_i(An + h_i^\prime)=o(1).\] Making the change of variables $m=An$, and expanding \[1_{m\equiv 0 \, \text{mod }A}=\frac{1}{A}\sum_{j=1}^Ae(jm/A),\] we reduce matters to showing that \[ \E_{m \leqslant AX}^{\log} e((\gamma^\prime + j/A)m) \prod\limits_{i=1}^k f_i(m + h_i^\prime)= o(1) \] for all integers $1 \leqslant j \leqslant A$. But as $\gamma^\prime + j/A$ is irrational, this follows from Lemma~\ref{le_tao_fh}(3). 
\end{proof}

\subsection{The pretentious case}

We now prove that Theorems~\ref{thm_twopoint}(1) and \ref{thm_hom99}(1) hold in the case of pretentious functions. 

\begin{lemma}\label{le_pretentious} Let $k\geq 1$ and let $f_1,\ldots,f_k:\mathbb{N}\to [-1,1]$ be pretentious multiplicative functions. Let $\alpha_1,\ldots, \alpha_k>0$ and $\beta_1,\ldots, \beta_k\in \mathbb{R}$ be such that $1,\alpha_1,\ldots, \alpha_k$ are linearly independent over $\mathbb{Q}$. Then we have
\begin{align}\label{eq_pretentious}
\lim_{X\to \infty} \E_{n\leqslant X}^{\log} \prod\limits_{i=1}^k f_i(\lfloor \alpha _i n +\beta_i\rfloor)=\prod_{i=1}^k \lim_{X\to \infty}\E_{n\leqslant X}^{\log} f_i(n).
\end{align}
\end{lemma}

\begin{proof}
From \cite[Theorem 6]{daboussi-delange} it follows that $f_i$ is almost periodic in the following sense: for any $\varepsilon>0$ there exist a decomposition
\begin{align*}
 f_i(n)=T_{\varepsilon,i}(n)+\mathcal{E}_{\varepsilon,i}(n),   
\end{align*}
where $T_{\varepsilon,i}(x)=\sum_{1\leq \ell\leq L_{\varepsilon,i}}c_{\varepsilon,i}(\ell)e(\gamma_{\ell,\varepsilon,i}x)$ for some $L_{\varepsilon,i}$, some real numbers $c_{\varepsilon,i}(\ell)$ and some rational numbers $\gamma_{\ell,\varepsilon,i}$, and $\limsup_{X\to \infty}\E_{n\leq X}^{\log}|\mathcal{E}_{\varepsilon,i}(n)|\leq \varepsilon$. Therefore, it suffices to prove for any rational numbers $\gamma_i$ that
\begin{align*}
\lim_{X\to \infty} \E_{n\leqslant X}^{\log} \prod\limits_{i=1}^k e(\gamma_i\lfloor \alpha _i n +\beta_i\rfloor)=\prod_{i=1}^k \lim_{X\to \infty}\E_{n\leqslant X}^{\log} e(\gamma_i n).
\end{align*}

Let $\gamma_i=a_i/d_i$ with $a_i$ and $d_i\geq 1$ integers. By writing $e(\gamma_i m)$ as a linear combination of the indicators $1_{m\equiv c\pmod{d_i}}$, it suffices to show for any integers $c_i,d_i\geq 1$ that  
\begin{align*}
\lim_{X\to \infty} \E_{n\leqslant X}^{\log} \prod\limits_{i=1}^k 1_{\lfloor \alpha _i n +\beta_i\rfloor\equiv c_i\pmod{d_i}}=\prod_{i=1}^k \lim_{X\to \infty}\E_{n\leqslant X}^{\log} 1_{n\equiv c_i\pmod{d_i}}=\frac{1}{d_1\cdots d_k}.
\end{align*}
Observe that $\lfloor \alpha n+\beta\rfloor\equiv c\pmod d$ for $0\leq c<d$ is equivalent to $\{\frac{\alpha}{d} n+\frac{\beta}{d}\}\in [\frac{c}{d},\frac{c+1}{d})$. Hence, it sufices to show that 
\begin{align*}
\lim_{X\to \infty} \E_{n\leqslant X}^{\log} \prod\limits_{i=1}^k 1_{\{\frac{\alpha _i}{d_i} n +\frac{\beta_i}{d_i}\}\in [\frac{c_i}{d_i},\frac{c_i+1}{d_i})}=\frac{1}{d_1\cdots d_k}.
\end{align*}
But this follows from the Kronecker--Weyl theorem since the numbers $1,\alpha_1/d_1,\ldots, \alpha_k/d_k$ are linearly independent over $\mathbb{Q}$.
\end{proof}

\section{Proofs of Theorem~\ref{thm_bohrchowla} and Theorem~\ref{thm_highercorrelationsandBohr}}
Understanding the correlations of non-pretentious multiplicative functions restricted to Bohr sets is straightforward, given the previous lemmas. 
\begin{proof}[Proof of Theorem~\ref{thm_bohrchowla}] Let $B\in \mathcal{B}_{\textnormal{convex}}$ and $\varepsilon >0$. Let $f_1,f_2:\mathbb{N}\to [-1,1]$ be multiplicative with $f_1$ non-pretentious. For any $\gamma \in \mathbb{R}$ we have \begin{align*}
\lim_{X\to \infty}\mathbb{E}_{n\leq X}^{\log}f_1(a_1n+h_1)f_2(a_2n+h_2)e(\gamma n)=0 
\end{align*} 
by Lemma~\ref{le_frantzikinakis}. Therefore, from Corollary~\ref{cor_approx} and the triangle inequality, if $X$ is large enough depending on $\varepsilon$,
\begin{align*}
\mathbb{E}_{n\leq X}^{\log}f_1(a_1n+h_1)f_2(a_2n+h_2)1_{B}(n)=O(\varepsilon). \end{align*}
Since $\varepsilon$ was arbitrary, Theorem~\ref{thm_bohrchowla} follows. 
\end{proof}
\begin{proof}[Proof of Theorem~\ref{thm_highercorrelationsandBohr}]
Let $B\in \mathcal{B}_{\textnormal{convex}}$ and $\varepsilon >0$. Let $f_1,\ldots,f_k:\mathbb{N}\to [-1,1]$ be multiplicative with $f_1$ non-pretentious. By Lemma~\ref{le_approximation} we write \[1_B(n) = \sum\limits_{l \leqslant L_{\varepsilon}} c_{\varepsilon}(l) e(\gamma_{l,\varepsilon} n) + \sum\limits_{a \leqslant q} t_a 1_{n \equiv a \, \text{mod } q} + \mathcal{E}_{\varepsilon}(n),\] where $\limsup_{X\to \infty}\E_{n \leqslant X}|\mathcal{E}_{\varepsilon}(n)|\leq \varepsilon$, $\gamma_{l,\varepsilon} \notin \mathbb{Q}$ for all $l$, $\vert c_{\varepsilon}(l)\vert \ll_{\varepsilon} 1$, $t_a \geqslant 0$ for all $a$, and $\frac{1}{q}\sum_{a \leqslant q} t_a = \delta_B + O(\varepsilon)$.  Parametrising the progression $n \equiv a \, \text{mod } q$, and using partial summation to pass from $\E_{n \leqslant X} \vert \mathcal{E}_{\varepsilon}(n)\vert$ to $\E_{n \leqslant X}^{\log} \vert \mathcal{E}_{ \varepsilon}(n)\vert$, we have
 \begin{align*}
 &\vert \E_{n \leqslant X}^{\log} 1_B(n) \prod\limits_{i=1}^k f_i(a_in + h_i)\vert\\
  &\leqslant \sum_{l\leq L_{\varepsilon}} \vert c_{\varepsilon}(l)\vert \vert \E_{n\leqslant X}^{\log} e(\gamma_{l,\varepsilon} n) \prod_{i=1}^k f_i(a_in + h_i) \vert + \sum\limits_{a \leqslant q}  \frac{t_a}{q} \vert\E_{m \leqslant \frac{X}{q}}^{\log} \prod_{i=1}^k f_i(a_i(qm + a) + h_i)\vert +  O(\varepsilon).
  \end{align*}
\noindent By combining the different parts of Lemma~\ref{le_frantzikinakis}, using critically the fact that $\gamma_{l,\varepsilon} \notin \mathbb{Q}$, there is some $\eta>0$ (fixed, independently of $X$ and $\varepsilon$) for which the above is \[ \leqslant o_{\varepsilon}(1) + \delta_B(1 - 2\eta) + O(\varepsilon).\] Picking $\varepsilon$ small enough and $X$ large enough, we obtain an upper bound of $\delta_B(1 - \eta)$ as required. 
\end{proof}

\section{Proof of Theorem~\ref{thm_twopoint}(1)--(2)}

By Lemma~\ref{le_pretentious}, we have Theorem~\ref{thm_twopoint}(1) in the case where $f_1,f_2$ are pretentious. We shall show that if $f_2$ is non-pretentious, then Theorem~\ref{thm_twopoint}(1) holds under the weaker assumption that $\alpha_1/\alpha_2$ is irrational. 

By the fact that $f_2$ is non-pretentious and real-valued, we have 
$$\sum_{p}\frac{1-\textnormal{Re}(f_2(p)\overline{\chi}(p)p^{-it})}{p}=\infty$$
for any real number $t$ and Dirichlet character $\chi$ (see~\cite[Lemma C.1]{MRT}). Hence, we have $\lim_{X\to \infty}\E_{n\leq X}^{\log}f_2(n)=0$ by Hal\'asz's theorem (\cite[Theorem 4.5 in Section III.4]{tenenbaum}). Now it suffices to show that
\begin{align*}
\lim_{X\to \infty} \E_{n\leqslant X}^{\log} f_1(\lfloor \alpha_1 n+\beta_1 \rfloor)f_2(\lfloor \alpha_2 n+\beta_2 \rfloor)=0.
 \end{align*}
Once we have shown this, Theorem~\ref{thm_twopoint}(2) also follows. 

We first reduce the correlation in Theorem~\ref{thm_twopoint}(1) to simpler correlations of the form
\begin{align*}
\mathbb{E}_{n\leq X}^{\log}f_1(n)f_2(\lfloor \alpha n+\beta \rfloor)1_{B}(n),    
\end{align*}
where $B\in \mathcal{B}_{\textnormal{convex}}$ is a Bohr set. To this end, we begin with the following lemma.

\begin{lemma}\label{le_bohr1}
Fix $\alpha_1,\alpha_2>0$ and $\beta_1,\beta_2\in \mathbb{R}$, and suppose that $\alpha_1/\alpha_2$ is irrational.  Then, there exist $M\in \mathbb{N}$ and linear polynomials $L_1,\ldots, L_M:\mathbb{R}\to \mathbb{R}$ of the form $L_i(x)=(\alpha_2/\alpha_1)x+n_i$ with $n_i\in \mathbb{Z}$ and a partition $A_1\sqcup A_2\sqcup \cdots \sqcup A_M$ of $\mathbb{N}$ such that
 \begin{enumerate}
    
    \item For any $1\leq i\leq M$, we have
 \begin{align*}
\lfloor \alpha_2 n+\beta_2\rfloor=\lfloor L_i(\lfloor \alpha_1 n+\beta_1\rfloor)\rfloor\quad \textnormal{whenever}\quad n\in A_i.     \end{align*}
\item For any $1\leq i\leq M$ and $\varepsilon>0$, there exist $J_{\varepsilon}\geq 1$, Bohr sets $B_{i,j,\varepsilon}\in \mathcal{B}_{2,\textnormal{convex}}$ for $j \leqslant J_{\varepsilon}$, and a decomposition
\begin{align*}
1_{A_i}(n)=\sum_{j\leq J_{\varepsilon}}1_{B_{i,j,\varepsilon}}(n)+\mathcal{E}_{i,\varepsilon}(n),    
\end{align*}
 where
 \begin{align*}
\limsup_{X\to \infty}\mathbb{E}_{n\leq X}|\mathcal{E}_{i,\varepsilon}(n)|\leq \varepsilon.     
 \end{align*}
  \end{enumerate}
\end{lemma}

\begin{proof}
Let $\gamma=\alpha_2/\alpha_1$. Write
\begin{align}\label{eq:alphabeta}
\alpha_2n+\beta_2=\gamma\lfloor \alpha_1 n+\beta_1\rfloor+r_n,
\end{align}
where
\begin{align}\label{eq:rn}
r_n=\beta_2-\gamma\beta_1+\gamma\{\alpha_1 n+\beta_1\}. 
\end{align}
We have $|r_n|\leq R$ for all $n$ for some $R\ll_{\alpha_i,\beta_i}1$. Therefore, for each $n$ there exists an integer $i\in [-R,R]$ such that 
\begin{align*}
 \lfloor\alpha_2n+\beta_2 \rfloor=\lfloor \gamma\lfloor \alpha_1 n+\beta_1\rfloor+r_n\rfloor=\lfloor \gamma\lfloor \alpha_1 n+\beta_1\rfloor+i\rfloor.  
\end{align*}
Now let $L_i(x):=\gamma x+i$. Consider the sets
\begin{align*}
A_i:&=\{n:\,\, \lfloor \alpha_2 n+\beta_2\rfloor=\lfloor L_i(\lfloor \alpha_1 n+\beta_1\rfloor)\rfloor\}.    
\end{align*}
The sets $A_i$ form a partition of $\mathbb{N}$, and note that by~\eqref{eq:alphabeta},~\eqref{eq:rn} we have
\begin{align*}
A_i&=\{n:\,\, \lfloor \alpha_2 n+\beta_2\rfloor=\lfloor (\alpha_2n+\beta_2)+i+\gamma\beta_1-\gamma\{\alpha_1n+\beta_1\}\rfloor\}\\
&=\{n:\,\, -\{ \alpha_2 n+\beta_2\}\leq i + \gamma\beta_1-\gamma\{\alpha_1n+\beta_1\}<1-\{ \alpha_2 n+\beta_2\}\},
\end{align*}
where we used the fact that $\lfloor x+y\rfloor=\lfloor x\rfloor$ if and only if $-\{x\}\leq y<1-\{x\}$.

Now, let $\varepsilon>0$ and let $K\geq 1$ be large in terms of $\varepsilon$. For brevity, write $u_i=i+\gamma\beta_1$. Then we can write
\begin{align*}
1_{A_i}(n)&=\sum_{0\leq k\leq K-1} 1_{\alpha_2n+\beta_2\in [k/K,(k+1)/K)\, \text{mod }1}1_{u_i-\gamma\{\alpha_1n+\beta_1\}\in (-k/K,1-k/K) }\\
&+O\left(1_{u_i-\gamma\{\alpha_1n+\beta_1\}-\alpha_2n-\beta_2\in [-1/K,1/K]\, \text{mod } 1}\right).    
\end{align*}
Each term inside the $k$ sum can be written as the sum of indicator functions of elements of $\mathcal{B}_{2,\textnormal{convex}}$. Moreover, since $\gamma$ is irrational,  by the Kronecker--Weyl theorem we have
\begin{align}
\label{eq_Weylerror}
 \limsup_{X\to \infty}\mathbb{E}_{n\leq X}1_{u_i-\gamma\{\alpha_1n+\beta_1\}-\alpha_2n-\beta_2\in [-1/K,1/K]\mod 1}=o_{K\to \infty}(1).    
\end{align}
Indeed, expressing $\{\alpha_1n + \beta_1\} = \alpha_1n + \beta_1 - \lfloor \alpha_1n + \beta_1 \rfloor$, it is enough to show that for any interval $I$ modulo 1 with length $O(1/K)$, \[ \limsup_{X \rightarrow \infty} \E_{n \leqslant X} 1_{\gamma \lfloor \alpha_1 n + \beta_1 \rfloor \in I \, \text{mod } 1} = o_{K \rightarrow \infty}(1).\] But since $\alpha_1 > 0$ the sequence $(\lfloor \alpha_1 n + \beta_1 \rfloor)_{n \leqslant X}$ contains integers at most $\alpha_1 X + \beta_1$ and at least $\lfloor \beta_1 \rfloor$, and the multiplicity of the sequence is at most $\lfloor \alpha_1^{-1} \rfloor + 1$. Therefore \[\E_{n \leqslant X} 1_{ \gamma \lfloor \alpha_1 n + \beta_1 \rfloor \in I \, \text{mod } 1} \ll \E_{n \leqslant \alpha_1 X} 1_{\gamma n \in I \, \text{mod } 1} + O(1/X) \ll \frac{1}{K}\] by Kronecker--Weyl (for large enough $X$). 

Thus~\eqref{eq_Weylerror} holds and the claim follows.
\end{proof}

Applying Lemma~\ref{le_bohr1}, we can write
\begin{align*}
&\mathbb{E}_{n\leq X}^{\log}f_1(\lfloor \alpha_1 n+\beta_1\rfloor) f_2(\lfloor \alpha_2 n+\beta_2\rfloor)\nonumber\\
&=\sum_{i\leq M}\sum_{j\leq J}\mathbb{E}_{n\leq X}^{\log}f_1(\lfloor \alpha_1 n+\beta_1\rfloor) f_2(\lfloor L_i(\lfloor \alpha_1 n+\beta_1\rfloor)\rfloor)1_{B_{i,j,J}}(n)+o_{X,J\to \infty}(1)
\end{align*}
for some Bohr sets $B_{i,j,J}\in \mathcal{B}_{2,\textnormal{convex}}$ and some linear polynomials $L_i:\mathbb{R}\to \mathbb{R}$ having leading coefficient $\alpha_2/\alpha_1$. Hence, it suffices to show that
\begin{align}\label{eq2}
\mathbb{E}_{n\leq X}^{\log}f_1(\lfloor \alpha_1 n+\beta_1\rfloor) f_2(\lfloor L(\lfloor \alpha_1 n+\beta_1\rfloor)\rfloor)1_{B}(n)=o(1)
\end{align}
for any $B\in \mathcal{B}_{\textnormal{convex}}$ and any polynomial $L(x)=\theta x+j$ with $j\in \mathbb{Z}$, where $\theta=\alpha_1/\alpha_2$.

For any $B\in \mathcal{B}_{\textnormal{convex}}$ and $\gamma \in \mathbb{R}$, introduce a multiplicity counting function
\begin{align*}
N_{B,\alpha,\beta,\gamma}(m):= \sum\limits_{n \in B: \,\, m=\lfloor \alpha n+\beta\rfloor} e(\gamma n).   
\end{align*}
Then, making a change of variables, we can rewrite the left-hand side of~\eqref{eq2} as
\begin{align*}
 \mathbb{E}_{m\leq \alpha_1 X}^{\log}f_1(m) f_2(\lfloor L(m)\rfloor)N_{B,\alpha_1,\beta_1,0}(m)+o(1).  
\end{align*}
We then need the following lemma on the structure of $N_{B,\alpha,\beta, \gamma}(m)$ (which is a version of Corollary~\ref{cor_approx} for $N_{B,\alpha,\beta, \gamma}(m)$). 

\begin{lemma}\label{le_bohr2} Fix $B\in \mathcal{B}_{\textnormal{convex}}$, $\alpha>0$ and $\beta, \gamma\in \mathbb{R}$. Then, for any $\varepsilon>0$, there exists some $K_{\varepsilon}\geq 1$, some sequence of real numbers $(\gamma_{k,\varepsilon})_{k\geq 1}$ and some complex numbers $c_{\varepsilon}(k)$ with $|c_{\varepsilon}(k)|\ll_{\varepsilon} 1$ such that for all $m \in \mathbb{Z}$
\begin{align*}
N_{B,\alpha,\beta,\gamma}(m)=\sum_{1\leq k\leq K_{\varepsilon}} c_{\varepsilon}(k)e(\gamma_{k,\varepsilon}m) + \mathcal{E}_{\varepsilon}(m) 
\end{align*}
and $\limsup_{X\to \infty}\mathbb{E}_{m\leq X}|\mathcal{E}_{\varepsilon}(m)|\leq \varepsilon$. 
\end{lemma}

\begin{proof}
Note that there exists an integer $N\geq 0$ such that
\begin{align*}
\left|\left[\frac{m-\beta}{\alpha},\frac{m+1-\beta}{\alpha}\right)\cap \mathbb{Z}\right|\in \{N,N+1\}    
\end{align*}
for all $m \in \mathbb{Z}$. Let $A_1$ be the set of $m$ such that $|[(m-\beta)/\alpha,(m+1-\beta)/\alpha)\cap \mathbb{Z}|=N$, and let $A_2$ be the complement of this set.

We can write
\begin{align*}
N_{B,\alpha,\beta,\gamma}(m)&=\sum_{(m-\beta)/\alpha\leq n<(m+1-\beta)/\alpha}1_{B}(n)e(\gamma n),
\end{align*}
and this equals
\begin{align*}
\sum_{0\leq j\leq N-1}1_{A_1}(m)1_{B}(\lceil (m-\beta)/\alpha\rceil +j) e(\gamma (\lceil (m-\beta)/\alpha\rceil +j)) \\+\sum_{0\leq j\leq N}1_{A_2}(m)1_{B}(\lceil (m-\beta)/\alpha\rceil +j) e(\gamma (\lceil (m-\beta)/\alpha\rceil +j)).
\end{align*}
The claim will follow if we can show that the four functions $m \mapsto 1_{A_1}(m)$, $m \mapsto 1_{A_2}(m)$, $m \mapsto 1_{B}(\lceil (m-\beta)/\alpha\rceil +j)$ and $m \mapsto e(\gamma(\lceil (m-\beta)/\alpha\rceil +j))$ can each be approximated by trigonometric polynomials of length $O_{\varepsilon}(1)$ with bounded coefficients (up to an error term which is $O(\varepsilon)$ in the normalised $L^1$ norm on the interval $[1,X] \cap \mathbb{Z}$). 

First note that the sets $A_i$ are both disjoint unions of elements of $\mathcal{B}_{1,\textnormal{convex}}$ (in fact, they are unions of sets of the form $\{m:\,\, \left\{\frac{m-\beta}{\alpha}\right\}\in I_i\}$ for some intervals $I_i$). Corollary~\ref{cor_approx} then means that $1_{A_i}$ can be suitably approximated. Next observe that by applying Corollary~\ref{cor_approx} to $B$ one reduces the task of approximating the term $m \mapsto 1_{B}(\lceil (m-\beta)/\alpha\rceil +j)$ to approximating terms of the form $m \mapsto e(\gamma(\lceil (m-\beta)/\alpha\rceil +j))$ (for arbitrary $\gamma$).

To achieve this, we write \[ e(\gamma ( \lceil (m-\beta)/\alpha \rceil + j)) = e(\gamma j)e(\gamma \frac{m-\beta}{\alpha}) e(\gamma \Big\{ \frac{m-\beta}{\alpha} \Big\}),\] which reduces matters to decomposing $e(\gamma \Big\{ \frac{m-\beta}{\alpha} \Big\})$. Then observe that for a suitably large integer $L\geq \varepsilon^{-1}$, for any $\gamma,\gamma_1,\gamma_2 \in \mathbb{R}$ we have
\begin{align*}
e(\gamma \{\gamma_1m+\gamma_2\})=\sum_{0\leq \ell<L}e\left(\gamma \frac{\ell}{L}\right)1_{\{\gamma_1m+\gamma_2\}\in [\ell/L,(\ell+1)/L)}+O(\varepsilon)
\end{align*}
Thus, up to an acceptable error, we can write $e(\gamma \Big\{ \frac{m-\beta}{\alpha} \Big\})$ as a bounded $\mathbb{C}$-linear combination of indicator functions of Bohr sets in $\mathcal{B}_{\textnormal{convex}}$. Applying Corollary~\ref{cor_approx} to each of these Bohr sets, the result follows. 
\end{proof}

Applying Lemma~\ref{le_bohr2} to~\eqref{eq2}, and writing out $L(m) = \theta m + j$, we reduce matters to proving that 
\begin{align}\label{eq9}
\sup_{ \gamma}\lim_{X\to \infty}\left|\mathbb{E}_{m\leq \alpha_1 X}^{\log}f_1(m)f_2(\lfloor \theta m + j\rfloor)e(\gamma m)\right|=0 
\end{align}
We are now in a position to apply the orthogonality criterion of K\'atai--Bourgain--Sarnak--Ziegler~\cite{BSZ} for multiplicative functions. 

\begin{lemma}[Orthogonality criterion]\label{le_katai} Let $a:\mathbb{N}\to \mathbb{C}$ be a bounded sequence of complex numbers. Suppose that, for any $\varepsilon>0$, there exists $P\geq 1$ such that for any primes $P\leq p<q$, we have
\begin{align}\label{eq_katai1}
\limsup_{X\to \infty}\left|\mathbb{E}_{n\leq X}^{\log}a(pn)\overline{a(qn)}\right|\leq \varepsilon.  
\end{align}
Then, for any $1$-bounded multiplicative function $f:\mathbb{N}\to \mathbb{C}$, we have 
\begin{align}\label{eq_katai2}
\lim_{X\to \infty}\mathbb{E}_{n\leq X}^{\log}f(n)a(n)=0.    
\end{align}
\end{lemma}

\begin{proof}
This can be deduced from~\cite[Lemma 2.16]{CHRST}. For the sake of completeness, we give a proof.

Suppose that $\varepsilon>0$ is small, $X$ is large enough in terms of $\varepsilon$, and $|\mathbb{E}_{n\leq X}^{\log}f(n)a(n)|\geq \varepsilon$. Let $Q$ be large enough in terms of $\varepsilon$ and $P$. By Elliott's inequality~\cite[Lemma 4.7]{elliott}, we have 
\begin{align*}
\mathbb{E}_{n\leq X}^{\log}f(n)a(n)=\frac{1}{\log\log Q} \sum_{p \leqslant X} \frac{1}{p}\mathbb{E}_{n\leq X}^{\log}f(pn)a(pn)+o_{Q\to \infty}(1).   
\end{align*}
Since $Q$ is large enough in terms of $\varepsilon$, the error term here is at most $\varepsilon/10$ in absolute value.
By the multiplicativity of $f$, we have $f(pn)=f(p)f(n)+O(1_{p\mid n})$, so we conclude that
\begin{align*}
\Big\vert\frac{1}{\log\log Q} \sum_{p \leqslant X} \frac{f(p)}{p}\mathbb{E}_{n\leq X}^{\log}f(n)a(pn) \Big\vert\geq \frac{4}{5}\varepsilon,  
\end{align*}
say. 

 Let $J=\lceil 10\varepsilon^{-2}\rceil$. Then, by the pigeonhole principle and the assumption that $Q$ is large, there exist distinct primes $P\leq p_1,\ldots, p_J\leq Q$ such that
\begin{align*}
\Big\vert\mathbb{E}_{n\leq X}^{\log}f(n)a(p_jn)\Big\vert\geq \frac{\varepsilon}{2}    
\end{align*}
for all $1\leq j\leq J$. Hence, there exist some unimodular complex numbers $c_j$ such that
\begin{align*}
 \sum_{j\leq J}c_j\mathbb{E}_{n\leq X}^{\log}f(n)a(p_jn)\geq \frac{\varepsilon J}{2}.    
\end{align*}
Exchanging the order of summation and then applying Cauchy--Schwarz, we deduce
\begin{align*}
\mathbb{E}_{n\leq X}^{\log}\Big|\sum_{j\leq J}c_ja(p_j n)\Big|^2\geq \frac{(\varepsilon J)^2}{4}. 
\end{align*}
Opening the square and separating the diagonal contribution, we obtain
\begin{align*}
  \sum_{\substack{i,j\leq J\\i\neq j}}c_i\overline{c_j}\mathbb{E}^{\log}_{n\leq X}a(p_i n)\overline{a(p_j n)}\geq \frac{(\varepsilon J)^2}{4}-J.  
\end{align*}
But recalling our choice of $J$, we obtain a contradiction with~\eqref{eq_katai1} (with $\varepsilon^2/8$ in place of $\varepsilon$). 
\end{proof}

By Lemma~\ref{le_katai}, to prove
\eqref{eq9} it suffices to show that for all fixed primes $p,q$ with $P \leqslant p < q$, that
\begin{align}\label{eq10}
\sup_{\gamma}\limsup_{X\to \infty}\left|\mathbb{E}_{n\leq X}^{\log}f_2(\lfloor p\theta n+j\rfloor)f_2(\lfloor q\theta n+j\rfloor)e(\gamma n)\right| =o_{P\rightarrow \infty}(1).
\end{align}
We continue with a lemma connecting $\lfloor p \theta n + j \rfloor$ and $\lfloor q \theta n + j\rfloor$ (in a similar spirit to Lemma~\ref{le_bohr1}). 
\begin{lemma}
\label{le_rationalratioBohr}
For all integers $p,q \geqslant 1$ and reals $\theta, \beta_p,\beta_q$, we have a finite partition $\mathbb{Z}=\mathcal{B}_1\sqcup \mathcal{B}_2\sqcup \cdots \sqcup \mathcal{B}_M$ such that $\mathcal{B}_i\in \mathcal{B}_{1,\textnormal{convex}}$ with
\begin{align*}
\lfloor q\theta n+\beta_q\rfloor=\frac{q\lfloor p\theta n+\beta_p\rfloor+r_i}{p}\quad \textnormal{whenever}\quad n\in \mathcal{B}_i    
\end{align*}
for some integers $r_i$. Furthermore, the phase of each $\mathcal{B}_i$ is $\theta$. 
\end{lemma} 
\begin{proof}
We have\[p \lfloor q \theta n + \beta_q \rfloor - q \lfloor p \theta n + \beta_p \rfloor = p\beta_q - q \beta_p + q \{ p \theta n + \beta_p\} - p \{q \theta n + \beta_q\}.\] For $i,j \in \mathbb{Z}_{ \geqslant 0}$ we define \[B_{i,j} = \{n \in \mathbb{Z}: \, \{p \theta n + \beta_p\} = p\{\theta n\} + \beta_p -i, \, \{q \theta n + \beta_q \} = q \{\theta n \} + \beta_q - j\}.\] The $B_{i,j}$ form a partition of $\mathbb{Z}$, all but finitely many of the $B_{i,j}$ are empty, and each $B_{i,j}$ is a union of finitely many sets $\mathcal{B} \in \mathcal{B}_{1,\textnormal{convex}}$ with phase $\theta$; for example, sets of the form \[ \mathcal{B} = B_1(\theta, U_{k,l}), \qquad \text{where  }\, U_{k,l} = \Big[\frac{k}{p} , \frac{k+1 - \{\beta_p\}}{p} \Big)  \cap \Big[\frac{l}{q} , \frac{l+1 - \{\beta_q\}}{q} \Big) \] for integers $k \in [0,p-1]$ and $l \in [0,q-1]$. If $n \in B_{i,j}$, from the above formulas we have \[  p \lfloor q \theta n + \beta_q \rfloor - q \lfloor p \theta n + \beta_p \rfloor = pj - qi \in \mathbb{Z}.\] The claim follows. 
\end{proof}
Applying Lemma~\ref{le_rationalratioBohr}, we have reduced~\eqref{eq10} to showing that for all integers $r$, all $B\in \mathcal{B}_{\textnormal{convex}}$, and all pairs of distinct primes $p,q$ with $P \leqslant p < q$, we have
\begin{align}\label{eq6}
\sup_{ \gamma}\limsup_{X\to \infty}\left|\mathbb{E}_{n\leq X}^{\log}f_2(\lfloor p\theta n+j\rfloor)f_2(\frac{q\lfloor p\theta n+j\rfloor+r}{p})e(\gamma n)1_{B}(n)1_{q\lfloor p\theta n+j\rfloor+r\equiv 0\, (\operatorname{mod} \, p)}\right| =o_{P\rightarrow \infty}(1).   
\end{align}
It is simple to control the $r=0$ case. Indeed, note that $r=0$ implies
\begin{align*}
\lfloor p\theta n+j\rfloor\equiv 0\pmod p,    
\end{align*}
or equivalently
\begin{align*}
\theta n\in \left[\frac{-j}{p},\frac{1-j}{p}\right)  \mod 1. 
\end{align*}
Since $\theta$ is irrational, the Kronecker--Weyl theorem~\cite[Exercise 1.1.5]{tao12} tells us that this happens for $(1/p+o(1))X=o_{P\to \infty}(X)$ integers $n\leq X$. The contribution of such $n$ can be bounded trivially by the triangle inequality.

It remains to consider $r \neq 0$. We prove the following general result, as we will need to refer to it several times before the end of the paper. 
\begin{lemma}
\label{le_usedeverywhere}
Let $p,q \geqslant 1$ be coprime integers, $\beta \in \mathbb{R}$, $\theta >0$, $r$ a non-zero integer, and $B \in \mathcal{B}_{ \textnormal{convex}}$. Then, for any non-pretentious multiplicative function $f:\mathbb{N}\to [-1,1]$, we have 
\begin{equation}
\label{eq_6c}
\sup_{ \gamma}\limsup_{X\to \infty}\left|\mathbb{E}_{n\leq X}^{\log}f(\lfloor p\theta n+\beta\rfloor)f(\frac{q\lfloor p\theta n+\beta\rfloor+r}{p})e(\gamma n)1_{B}(n)1_{q\lfloor p\theta n+\beta\rfloor+r\equiv 0 \,\, \operatorname{mod} \, p}\right| =0. 
\end{equation}
\end{lemma}
\begin{proof}

Recalling that
\begin{align*}
N_{B,p\theta, \beta, \gamma}(m):=\sum\limits_{n \in B:\,\, m=\lfloor p\theta n+\beta\rfloor} e(\gamma n),    
\end{align*}
we rewrite \eqref{eq_6c} as 

\begin{align}\label{eq7}
\sup_{ \gamma}\limsup_{X\to \infty}\left|\mathbb{E}_{m\leq p\theta X}^{\log}f(m)f(qm+r)N_{B,p\theta,\beta,\gamma}(m)1_{qm+r\equiv 0\pmod p}\right|=0.  
\end{align}
By Lemma~\ref{le_bohr2}, we express $N_{B,p\theta, \beta, \gamma}(m)$ as a trigonometric polynomial up to small error. We also expand the condition $m \equiv -r \overline{q}\, \operatorname{mod} \, p$ by the exponential sum $$\frac{1}{p}\sum_{1\leq a\leq p}e(a(m + r\overline{q})/p).$$  It therefore suffices to show that
\begin{align}\label{eq8c}
 \sup_{\gamma}\limsup_{X\to \infty}\left|\mathbb{E}_{m\leq p\theta X}^{\log}f(m)f(qm+r)e(\gamma m)\right|=0.     
\end{align}
But this follows from Lemma~\ref{le_frantzikinakis} (since $r\in \mathbb{Z}\setminus \{0\}$). Thus the lemma has been proved.
\end{proof}

Applying Lemma~\ref{le_usedeverywhere} to expression~\eqref{eq6}, Theorem~\ref{thm_twopoint}(1) follows. As already remarked, the argument settled Theorem \ref{thm_twopoint}(2) as well.  \qed

\section{Proof of Theorem~\ref{thm_twopoint}(3)}
Since $\alpha_1/\alpha_2$ is rational, there are coprime positive integers $p$ and $q$ and real $\theta$ for which $\alpha_1 = p \theta$ and $\alpha _2 = q \theta$. By Lemma~\ref{le_rationalratioBohr}, there is an integer $J$ and a partition $\mathbb{Z}=B_{-J}\sqcup B_{-J+1}\sqcup \cdots \sqcup B_J$ such that $B_j$ is a disjoint union of Bohr sets in $\mathcal{B}_{1,\textnormal{convex}}$ with phase $\theta$ with
\begin{align*}
\lfloor q\theta n+\beta_2\rfloor=\frac{q\lfloor p\theta n+\beta_1\rfloor + j}{p}\quad \textnormal{whenever}\quad n\in B_j.    
\end{align*}

We claim that if $j \neq 0$ then
\begin{align}\label{eq_Aj}
|\mathbb{E}_{n\leq X}^{\log}\lambda(\lfloor \alpha_1 n+\beta_1\rfloor)\lambda(\lfloor \alpha_2 n+\beta_2\rfloor)1_{B_j}(n)|=o(1).    
\end{align}
Indeed, writing $B_j$ as a disjoint union of elements of $\mathcal{B}_{\textnormal{convex}}$ it is enough to show that 
 \begin{align*}
 |\mathbb{E}_{n\leq X}^{\log}\lambda(\lfloor p \theta  n+\beta_1\rfloor)\lambda(\frac{q\lfloor p \theta n+\beta_1\rfloor+j}{p})1_{B}(n)1_{q\lfloor p \theta n+\beta_1\rfloor + j\equiv 0\pmod p}|=o(1).   
\end{align*}
for any $B\in \mathcal{B}_{\textnormal{convex}}$. But this result follows directly from Lemma~\ref{le_usedeverywhere}. 

Consider now the contribution from $B_0$, namely
\begin{equation}
\E_{n \leqslant X}^{\log} \lambda(\lfloor p \theta n + \beta_1\rfloor) \lambda( \frac{q \lfloor p \theta n + \beta_1 \rfloor}{p})1_{B_0}(n).
\end{equation}
Since $B_0$ is a disjoint union of finitely many sets in $\mathcal{B}_{1, \textnormal{convex}}$ (call these Bohr sets $S_1, \dots, S_M$) we have \begin{equation}
\lim_{X \rightarrow \infty}\E_{n \leqslant X}^{\log} \lambda(\lfloor p \theta n + \beta_1\rfloor) \lambda( \frac{q \lfloor p \theta n + \beta_1 \rfloor}{p})1_{B_0}(n) = \lambda(p) \lambda(q) \sum_{i \leqslant M} \delta_{S_i}. 
\end{equation}
Including the terms with $n \in B_j$, for $j \neq 0$, we have \[\lim_{X \rightarrow \infty} \E_{n \leqslant X}^{\log} \lambda( \lfloor \alpha_1n + \beta_1 \rfloor) \lambda(\lfloor \alpha_2 n + \beta_2 \rfloor) = \lambda(p) \lambda(q) \sum_{i \leqslant M} \delta_{S_i}.\] In particular the limit exists. Finally, observe that for any Bohr set $S_i \in \mathcal{B}_{1, \textnormal{convex}}$ the density $\delta_{S_i}$ is positive if and only if $S_i$ is infinite. Therefore $\sum_{i \leqslant M} \delta_{S_i} = 0$ if and only if $B_0$ is finite. This completes the proof of the second part of Theorem~\ref{thm_twopoint}. \qed

\begin{remark}
It is clear from the proof that one could prove a similar result with $\lambda$ replaced by any non-pretentious completely multiplicative function $f:\mathbb{N}\to [-1,1]$ such that $f(n)\neq 0$ for all $n\geq 1$.
\end{remark}

\section{Higher order correlations}

In this section we will prove Theorem \ref{thm_hom99}. By Lemma~\ref{le_pretentious}, we already have Theorem~\ref{thm_hom99} part (1) in the case where $f_1,\ldots, f_k$ are pretentious. Hence, we may assume in this section that $f_1$ is non-pretentious. Then we have $\lim_{X\to \infty}\E_{n\leq X}^{\log}f_1(n)=0$ by Hal\'asz's theorem, so it suffices to show that
\begin{align*}
\limsup_{X\to \infty} \Big\vert \E_{n\leqslant X}^{\log} \prod\limits_{i=1}^k f_i(\lfloor \alpha _i n \rfloor)\Big\vert \leqslant 1 - \eta.
 \end{align*}

\begin{proof}[Proof of Theorem \ref{thm_hom99} part (1)]
For contradiction we assume that  \[ \Big\vert \E_{n\leqslant X}^{\log} \prod\limits_{i=1}^k f_i(\lfloor \alpha _i n \rfloor)\Big\vert \geqslant 1 - \eta\] for some fixed $\eta>0$ and for arbitrarily large values of $X$. Therefore there exists some $u \in \{-1,+1\}$ and $S_1 \subset [X]$ for which \[ \E_{n \leqslant X}^{\log} 1_{S_1}(n) \geqslant 1 - O(\eta)\] and \[ \Big\vert \prod\limits_{i \leqslant k} f_i( \lfloor \alpha_1 n \rfloor) - u \Big\vert \ll \eta\] for all $n \in S_1$.

Let $r \geqslant 2$ be prime. Define \[ D_r: = (\frac{1}{r^2}, \frac{2}{r^2}) \times (\frac{1}{r}, \frac{1}{r} + \frac{1}{r^2})^{k-1} \subset [0,1)^k.\] Since $1,\alpha_1,\dots,\alpha_k$ are linearly independent over $\mathbb{Q}$, by the Kronecker-Weyl theorem we have that the Bohr set $B_r: = B(\alpha, D_r)$ has positive density $\delta_{B_r} = r^{-2k}$. We also have that for all $n \in B_r$, 
\begin{align*}
 \lfloor \alpha_1 r^2 n\rfloor &= r \lfloor \alpha_1 rn \rfloor + 1\\
 \lfloor \alpha_i r^2 n \rfloor & = r \lfloor \alpha_i rn \rfloor \qquad ( i \geqslant 2)\\
 \lfloor \alpha_i rn \rfloor &\neq 0 \text{ mod } r \qquad ( i \geqslant 2).
 \end{align*}
 
Observe that 
\begin{align*}
\E_{n \leqslant X}^{\log}1_{r \vert n} 1_{S_1}(n) \geqslant \frac{1}{r} - O(\eta) - o(1).
\end{align*}
Hence \[\E_{n \leqslant X/r} ^{\log} 1_{S_1}(rn) \geqslant 1 - O(r \eta) - o(1)\] and so \[\E_{n \leqslant X} ^{\log} 1_{S_1}(rn) \geqslant 1 - O(r \eta) - o(1).\] From this argument, letting \[ S_2: = B_r \cap \{n: \, rn \in S_1\} \cap \{n : \, r^2n \in S_1\},\] we see 
\begin{align*}
\E_{n \leqslant X}^{\log} 1_{S_2}(n)  \geqslant \delta_{B_r} - O(r^2 \eta) - o(1). 
\end{align*}
Then for $n \in S_2$ we have \[ u+O(\eta) = \prod\limits_{i \leqslant k} f_i( \lfloor \alpha_i rn\rfloor) = f_1(\lfloor \alpha_1 rn \rfloor) \prod\limits_{i=2}^k f_i( \lfloor \alpha_i r n \rfloor)\] and 
\begin{align*}
u+O(\eta) = \prod\limits_{i \leqslant k} f_i( \lfloor \alpha_i r^2n \rfloor) &= f_1(r \lfloor \alpha_1  rn\rfloor + 1) \prod\limits_{i=2}^{k} f_i(r \lfloor \alpha_i rn \rfloor) \\
 & =f_1(r \lfloor \alpha_1 rn\rfloor + 1) \prod\limits_{i=2}^kf_i(r)\cdot \prod\limits_{i=2}^k f_i( \lfloor \alpha_ir n \rfloor)
 \end{align*} 
by multiplicativity and the fact that $(\lfloor \alpha_i rn\rfloor,r)=1$ for all $i \geqslant 2$.  

Note that if for some $u\in \{-1,+1\}$ and some real numbers $|u_i|\leq 1$ we have $u+O(\eta)=u_1u_3$ and $u+O(\eta)=u_2u_3$, then $|u_1u_2-1|=O(\eta)$. Therefore,
 \begin{equation}
 \label{eq_finalcont}
 \vert\E_{n \leqslant X}^{\log} 1_{B_{r}}(n)f_1(\lfloor \alpha_1r n\rfloor) f_1( r \lfloor \alpha_1 rn \rfloor + 1) \vert \geqslant \delta_{B_{r}} - O(r^2 \eta) - o(1).
 \end{equation}
 
However, applying Lemma~\ref{le_usedeverywhere} with $\theta = r \alpha_1$ we have
 \begin{equation}
 \label{eq_2pointfromkpoint}
 \vert\E_{n \leqslant X}^{\log} 1_{B_{r}}(n) f_1(\lfloor \alpha_1 rn\rfloor) f_1( r \lfloor \alpha_1 rn \rfloor + 1) \vert = o(1).
 \end{equation}
Expressions~\eqref{eq_finalcont} and~\eqref{eq_2pointfromkpoint} are in contradiction for large enough $X$ and small enough $\eta$. This resolves Theorem~\ref{thm_hom99} part (1).
\end{proof}
\begin{proof}[Proof of Theorem \ref{thm_hom99} part (2)]
 Let $\mathcal{V} = \{v_1,\dots,v_{k-k^\prime}\}$ denote the maximal linearly independent set of vectors $\mathcal{V} \subset \mathbb{Z}^k$ from the hypotheses of the theorem. By the abelian Ratner's theorem of~\cite[Proposition 1.1.5]{tao12} we may write $(\alpha_1,\dots, \alpha_k) = \alpha'+\alpha''$, where $\alpha':=(\alpha_1^{\prime},\dots, \alpha_k^\prime) \in \mathbb{R}^k$, $\alpha''=(\alpha_1^{\prime\prime},\dots,\alpha_k^{\prime\prime}) \in \mathbb{Q}^k$, and the sequence $\alpha' n  \, \text{mod } \mathbb{Z}^k$ is totally equidistributed in a subtorus $T^\prime \leqslant \mathbb{T}^k$. We also have that the dimension of $T^\prime$ is $k^\prime$, and $T^\prime$ is the projection modulo $\mathbb{Z}^k$ of $\{u \in \mathbb{R}^k: \, v_i \cdot u = 0 \text{ for all } i\}$. Letting $q$ be the least common multiple of the denominators of the $\alpha_i^{\prime\prime}$, we have $\alpha qn \equiv \alpha^\prime qn \, \text{mod } \mathbb{Z}^k$ for all $n \in \mathbb{Z}$.

For contradiction we assume that  \[ \Big\vert \E_{n\leqslant X}^{\log} \prod\limits_{i=1}^k f_i(\lfloor \alpha _i n \rfloor)\Big\vert \geqslant 1 - \eta\] for some fixed $\eta>0$ and for arbitrarily large values of $X$. Using the same argument as in the previous proof, this implies that 
\[ \Big\vert \E_{n \leqslant X}^{\log} \prod\limits_{i=1}^k f_i(\lfloor \alpha _i qn \rfloor)\Big\vert \geqslant 1 - O(q\eta) - o(1).\] Therefore there exists some $u \in \{-1,+1\}$ and $S_1 \subset [X]$ for which 
$$\mathbb{E}_{n\leq X}^{\log}1_{S_1}(n) \geqslant 1 - O(q \eta) - o(1)$$
and $$\Big\vert\prod_{i \leqslant k} f_i(\lfloor \alpha _i qn \rfloor)-u\Big\vert\ll q \eta$$ 
for all $n \in S_1$.

Let $r \geqslant 2$ be prime, and let $w \in \mathbb{R}_{>0}^k$ be the vector from the hypotheses of the theorem. Write $w = (w_1,\dots,w_k)$ and assume without loss of generality that $w_1 >w_2 \geqslant w_i > 0$ for all $i=3,\dots,k$. Define \[ D_{q,r}: = T^\prime \cap \Big( (\frac{1}{qr}, \frac{2}{qr}) \times (0,\frac{1}{qr})^{k-1}\Big) \quad \text{mod } \mathbb{Z}^k.\] We claim that $D_{q,r}\neq \emptyset$. Indeed, since $w_1$ is strictly larger than $w_2$ we may choose $c \in \mathbb{R}$ satisfying \[c \in (\frac{1}{qr w_1}, \min(\frac{2}{qr w_1},\frac{1}{qr w_2})).\] Since $c w \cdot v_j = 0$ for all $j$, we conclude that $cw \text{ mod } \mathbb{Z}^k \in T^\prime$. But by assumptions on the sizes of the $w_i$, \[c w \in \Big( (\frac{1}{qr}, \frac{2}{qr}) \times (0,\frac{1}{qr})^{k-1}\Big).\] So $cw \text{ mod } \mathbb{Z}^k \in D_{q,r}$.

Thus $D_{q,r}$ is a non-empty open subset of $T^\prime$ in the subspace topology. Therefore, when $T^\prime$ is endowed with the normalised Haar measure $\mu$,  we have $ \mu(D_{q,r}) >0$. Since the sequence $\alpha^\prime n$ is totally equidistributed in $T^\prime$, we know that the Bohr set $B_{q,r} \in \mathcal{B}_{\textnormal{convex}}$ defined by \[B_{q,r}: = B\Big(\alpha^\prime, ( \frac{1}{qr}, \frac{2}{qr}) \times (0, \frac{1}{qr})^{k-1}\Big)\] is equal to $B( \alpha^\prime, D_{q,r})$ and has density $\delta_{B_{q,r}} = \mu(D_{q,r}) >0$. 

Let \[ S_2: = B_{q,r} \cap S_1 \cap \{n: \, rn \in S_1\}.\] Then, by the same argument we used to lower-bound $\E_{n \leqslant X}^{\log} 1_{S_1}(n)$, we conclude that \[\E_{n \leqslant X}^{\log} 1_{S_2}(n) \geqslant \delta_{B_{q,r}} - O(r q\eta) - o(1).\] Furthermore, using the fact that $\alpha qn \equiv \alpha^\prime qn \, \text{mod } \mathbb{Z}^k$ for all $n \in \mathbb{Z}$, for $n \in S_2$ we have
\begin{align*}
 \lfloor \alpha_1 qrn\rfloor&=r\lfloor \alpha_1qn\rfloor+1 \\ \lfloor \alpha_iqrn\rfloor&=r\lfloor \alpha_iqn\rfloor \qquad (2\leq i\leq k).
\end{align*}

Then for $n\in S_2$ we have \[ u+O(q\eta) = \prod\limits_{i \leqslant k} f_i( \lfloor \alpha_i qn\rfloor) = f_1(\lfloor \alpha_1 qn \rfloor) \prod\limits_{i=2}^k f_i( \lfloor \alpha_i q n \rfloor)\] and 
\begin{align*}
u+O(q\eta) = \prod\limits_{i \leqslant k} f_i( \lfloor \alpha_i qrn \rfloor) &= f_1(r \lfloor \alpha_1  qn\rfloor + 1) \prod\limits_{i=2}^{k} f_i(r \lfloor \alpha_i qn \rfloor) \\
 & =f_1(r \lfloor \alpha_1 qn\rfloor + 1) \prod\limits_{i=2}^kf_i(r)\cdot \prod\limits_{i=2}^k f_i( \lfloor \alpha_iq n \rfloor)
 \end{align*} 
by complete multiplicativity of $f_2,\dots,f_k$. Arguing analogously to the previous proof, we conclude that
 \begin{equation}
 \label{eq_finalconta}
 \vert\E_{n \leqslant X}^{\log} 1_{B_{q,r}}(n)f_1(\lfloor \alpha_1q n\rfloor) f_1( r \lfloor \alpha_1 qn \rfloor + 1) \vert \geqslant \delta_{B_{q,r}} - O(r q \eta) - o(1).
 \end{equation}
However, applying Lemma~\ref{le_usedeverywhere} with $\theta = q \alpha_1$ we have
 \begin{equation}
 \label{eq_2pointfromkpointa}
 \vert\E_{n \leqslant X}^{\log} 1_{B_{q,r}}(n) f_1(\lfloor \alpha_1 qn\rfloor) f_1( r \lfloor \alpha_1 qn \rfloor + 1) \vert = o(1).
 \end{equation}
Expressions~\eqref{eq_finalconta} and~\eqref{eq_2pointfromkpointa} are in contradiction for large enough $X$ and small enough $\eta$. This resolves Theorem~\ref{thm_hom99} part (2).
\end{proof}

\begin{remark}
Only the multiplicativity of $f_1$ and the complete multiplicativity of $f_2,\dots f_k$ at $r$ was used in the proof of Theorem \ref{thm_hom99}(2). Unfortunately the method only saves a value $\eta \ll q^{-k-1} r^{-k -1}$ over the trivial bound, and this seems to be not enough to remove the complete multiplicativity assumption using the device from the proof of Lemma \ref{le_frantzikinakis}. 
\end{remark}

\bibliographystyle{plain}
\bibliography{refsupdated}

\end{document}